\documentclass[preprint,1p]{elsarticle}

\usepackage{epsfig}
\usepackage{amssymb}
\usepackage{amsmath}
\usepackage{amsfonts}
\usepackage{amsthm}
\usepackage{latexsym}
\usepackage{bbm}
\usepackage{graphicx}
\usepackage{etoolbox}
\usepackage[utf8]{inputenc}

\usepackage{graphics,epsfig,color,enumerate}
\usepackage{setspace,mathtools,enumitem,dsfont,verbatim}
\usepackage{caption}
\usepackage{subcaption}
\usepackage[english]{babel}
\newtheorem{Thm}{Theorem}[section]
\newtheorem{Def}[Thm]{Definition}
\newtheorem{Prop}[Thm]{Proposition}

\newtheorem{Remark}[Thm]{Remark}
\AfterEndEnvironment{Thm}{\noindent\ignorespaces}
\AfterEndEnvironment{Def}{\noindent\ignorespaces}
\AfterEndEnvironment{Prop}{\noindent\ignorespaces}
\AfterEndEnvironment{Cor}{\noindent\ignorespaces}
\AfterEndEnvironment{Lemma}{\noindent\ignorespaces}
\AfterEndEnvironment{Remark}{\noindent\ignorespaces}
\AfterEndEnvironment{proof}{\noindent\ignorespaces}

\newcommand{\E}{\ensuremath{\mathbb{E}}}
\newcommand{\N}{\ensuremath{\mathbb{N}}}
\newcommand{\Z}{\ensuremath{\mathbb{Z}}}

\newcommand{\R}{\ensuremath{\mathbb{R}}}

\newcommand{\T}{\ensuremath{\mathbb{T}}}

\renewcommand{\P}{\ensuremath{\mathbb{P}}}

\newcommand{\e}{\epsilon}
\newcommand{\de}{\delta}
\numberwithin{equation}{section}

\makeatletter
\def\@captype{figure}
\makeatother

\begin{document}
	
	\begin{frontmatter}
		
		\title{Multi-colony Wright-Fisher with seed-bank}


		
		\author{Frank den Hollander\fnref{Leiden}}
		\author{Giulia Pederzani\fnref{Leiden}}
		\fntext[Leiden]{Mathematical Institute, Leiden University, P.O.\ Box 9512,
			2300 RA Leiden, The Netherlands.}
		
		\begin{abstract}
		We consider a multi-colony version of the Wright-Fisher model with seed-bank that 
		was recently introduced by Blath et al. Individuals live in colonies and change type 
		via resampling and mutation. Each colony contains a seed-bank that acts as a genetic 
		reservoir. Individuals can enter the seed-bank and become dormant or can exit the 
		seed-bank and become active. In each colony at each generation a fixed fraction of 
		individuals swap state, drawn randomly from the active and the dormant population. 
		While dormant, individuals suspend their resampling. While active, individuals resample 
		from their own colony, but also from different colonies according to a random walk 
		transition kernel representing migration. Both active and dormant individuals mutate.
		
		We derive a formula for the probability that two individuals drawn randomly from two 
		given colonies are identical by descent, i.e., share a common ancestor. This formula, 
		which is formulated in Fourier language, is valid when the colonies form a discrete 
		torus. We consider the special case of a symmetric slow seed-bank, for which in 
		each colony half of the individuals are in the seed-bank and at each generation the 
		fraction of individuals that swap state is small. This leads to a simpler formula, from 
		which we are able to deduce how the probability to be identical by descent depends 
		on the distance between the two colonies and various relevant parameters. Through 
		an analysis of random walk Green functions, we are able to derive explicit scaling 
		expressions when mutation is slower than migration. We also compute the spatial 
		second moment of the probability to be identical by descent for all parameters when 
		the torus becomes large. For the special case of a symmetric slow seed-bank, we 
		again obtain explicit scaling expressions.
		\end{abstract}
		
		
	\end{frontmatter}

\tableofcontents


\section{Introduction}
\label{sec:intro}


\subsection{Background}

The Wright-Fisher model with seed-bank aims at modelling the genetic evolution of populations 
in which individuals are allowed to achieve a dormant state and maintain it for many 
generations. Dormancy is understood as a reversible rest period, characterised by 
low metabolic activity and interruption of phenotypic development (see Lennon and Jones
\cite{microbialseedbanks}). This type of behaviour is observed in many taxa, including 
plants, bacteria and other micro-organisms, as a response to unfavourable environmental 
conditions. Dormant individuals can be resuscitated under more favourable conditions, after 
a varying and possibly large number of generations, and reprise reproduction. This strategy has 
been shown to have important implications for population persistence, maintenance of genetic 
variability and even stability of ecosystem processes, acting as a buffer against such evolutionary 
forces as genetic drift, selection and environmental variability. The importance of this evolutionary 
trait has led to attempts at modelling seed-banks from a mathematical perspective.

The mathematical modelling of seed-bank effects has been somewhat challenging. In fact, since 
individuals can remain dormant for arbitrarily many generations, it is problematic to retain the 
Markov property. The most successful models to date are extensions of the Wright-Fisher 
model, the classical model for genetic evolution of a population viewed from a probabilistic 
perspective. However, so far these extensions have been able to capture only weak seed-bank 
effects, resulting in a delayed Wright-Fisher model, or they have led to extreme behaviours 
that seem artificial.

The most relevant model was proposed by Kaj, Krone and Lascoux \cite{KKL}, who 
allowed individuals in a population of fixed size to select a parent from a random 
number $B$ of generations in the past, where $B$ is an $\mathbb{N}$-valued random variable 
assumed to be independent and identically distributed for each individual. The authors 
show that if $B$ is bounded, then after the usual rescaling of time by the population size the 
model converges to a delayed Kingman coalescent, where the rates are multiplied by 
$1/\E[B]^2$. More generally, it was proven by Blath, Gonz\'ales Casanova, Kurt and 
Span\`o \cite{longrange} that a sufficient condition for convergence to the Kingman 
coalescent is $\E[B]<\infty$. The sampling from past generations only leads to a delay 
in the coalescence process and leaves the coalescent structure unchanged. Therefore 
it is not capturing the role of seed-banks in maintaining genetic variability. One therefore 
speaks of \emph{weak} seed-bank effects.

A further extension of the model, which allows for \emph{strong} seed-bank effects, 
was proposed in \cite{longrange}. In particular, the authors study the case in which 
the age-distribution $\chi$ of $B$ is heavy-tailed, 
namely, 
\begin{equation}
\chi(B \geq n) = L(n)\,n^{-\alpha}, \quad n\in\N, \qquad 0<\alpha<\infty,
\end{equation} 
where $L$ is slowly varying as $n\to\infty$. It is proved that for $\alpha>\tfrac12$ the most 
recent common ancestor (MRCA) of two randomly sampled individuals exists with probability
$1$, but the expected time to the MRCA is infinite as soon as $\alpha<1$, while for 
$0<\alpha<\tfrac12$ with positive probability a MRCA does not even exist at all. 

Such extreme behaviour may seem artificial, and for this reason Gonz\'alez Casanova et al.\ 
in \cite{strongeffect} and Blath et al.\ in \cite{genealstrongseedbank} have studied the case in 
which $B$ scales with the total population size $N$. In particular, the age-distribution $\chi$ 
of $B$ is given by 
\begin{equation}
\chi = (1-\e)\,\de_1 + \e\,\de_{N^\beta}, \qquad 0 < \beta < \infty, \quad \e\in (0,1).
\end{equation} 
For $\beta<\tfrac13$, this model again shows convergence to the Kingman coalescent, but 
requires a rescaling of time by the non-classical factor $N^{1+2\beta}$. In other words, this
choice of $\chi$ considerably increases the expected time to the MRCA. Still, it leaves the 
coalescent structure unchanged, and results in other parameter regimes remain somewhat 
elusive.

In summary, the mathematical results to date have been effective in modelling weak seed-bank 
effects, but for strong seed-bank effects they have so far been unsatisfying in suggesting new 
limiting coalescent structures. Blath et al.\ \cite{seedbank} propose a new model in which
individuals can enter and exit a seed-bank at each generation. While in the seed-bank,
individuals \emph{suspend} their resampling and preserve their type. The main advantage of this 
model is that it retains the Markov property, while the long-term behaviour is in line with our 
intuition and does not require artificial scaling assumptions. It also offers a natural interpretation 
for the scaling limit of both the forward and the backward process, and the limiting genealogy 
is given by a new coalescent structure. 


\subsection{Outline}

The goal in the present paper is to study a \emph{multi-colony} version of the model in 
\cite{seedbank}, where individuals can \emph{migrate} between different colonies, each 
carrying a seed-bank, and are also subject to \emph{mutation}. The main quantity we are 
after is the probability that two individuals drawn randomly from two given colonies are 
\emph{identical by descent}, i.e., share a common ancestor. 

The remaining sections are organised as follows. In Section~\ref{sec:seedbankdef} we 
define the single-colony Wright-Fisher model with seed-bank. In Section~\ref{sec:stepstone} 
we present the multi-colony version of this model and derive a formula for the probability 
that two individuals drawn randomly from two given colonies are identical by descent. 
This formula (stated in Theorem~\ref{MainTheorem} below) comes in Fourier language 
and is valid when the colonies form a discrete torus. In Section~\ref{spec-choice-par} 
we consider the special case of a symmetric slow seed-bank, for which in each colony 
half of the individuals are in the seed-bank and at each generation the fraction of individuals 
that swap state is small. This leads to a simpler formula (stated in Theorems~\ref{thm:smalld} 
and \ref{thm:smalldalt} below), from which we are able to deduce how the probability to be 
identical by descent depends on the distance between the two colonies and various 
relevant parameters. Through an analysis of random walk Green functions, we are able 
to derive explicit scaling expressions when mutation is slower than migration (stated in 
Theorems~\ref{d0exp1}--\ref{d>0exp2} below). In Section~\ref{sec:2ndmom} we compute the spatial 
second moment of the probability to be identical by descent for all parameters when the 
torus becomes large (stated in Theorem~\ref{secondmoment} below). For the special 
case of a symmetric slow seed-bank, we again obtain explicit scaling expressions (stated
in Theorem~\ref{secondmomentscal} below).


\section{The Wright-Fisher model with seed-bank}
\label{sec:seedbankdef}

In Section~\ref{WF} we recall the standard Wright-Fisher model, the simplest model 
for population genetics, where the only evolutionary force at play is \emph{resampling}. 
In Section~\ref{WFS} we recall the extension of this model studied in \cite{seedbank}, 
namely, with a \emph{seed-bank}. In Section~\ref{sec:stepstone} we will introduce a
multi-colony version of the extended model, which will be our main object of study. In 
what follows, we write $\N=\{1,2,\ldots\}$ and $\N_0=\N\cup\{0\}$.


\subsection{Wright-Fisher model}
\label{WF}

Consider a population of $N$ haploid individuals, with $N$ fixed. For each genetic 
locus, an individual carries one copy of the gene at that locus that is assumed to be 
one of two types, denoted by $a$ and $A$. The model is discrete in time. At each 
time unit, every individual from the new generation chooses an individual from the 
previous population uniformly at random and adopts its type. This choice is 
independent of time and of the choices of the other individuals. This type of 
resampling mechanism is called parallel updating (all individuals choose an ancestor 
at the same time and independently of each other). Let
\begin{equation}
X_n = \text{number of individuals of type $a$ at time $n$.}
\end{equation}
The sequence $X=(X_n)_{n\in\N_0}$ is the discrete-time Markov chain with state 
space
\begin{equation}
\Omega=\{0,1,2,\ldots,N\}
\end{equation} 
and transition kernel
\begin{equation}
p_{ij} = \binom{N}{j} \left(\frac{i}{N}\right)^j \left(\frac{N-i}{N}\right)^{N-j}, \qquad i,j\in\Omega.
\end{equation}
The latter follows from the fact that, if at time $n$ there are $i$ individuals of type $a$, 
then there will be $j$ individuals at time $n+1$ if and only if precisely $j$ individuals choose 
an ancestor of type $a$ and $N-j$ individuals choose an ancestor of type $A$. The initial 
condition can be any state $X_0\in\Omega$.

The states $0$ and $N$ are traps: $p_{00} = p_{NN} = 1$. Consequently, one of the two 
genetic types eventually becomes extinct: genetic variability is lost through chance.


\subsection{Wright-Fisher model with seed-bank}
\label{WFS}

The model introduced in \cite{seedbank} consists of a haploid population of fixed size $N$ 
that reproduces in discrete generations. Each individual carries a genetic type from a generic 
type space $E$. In what follows we will focus on the bi-allelic case, $E=\{a,A\}$. Alongside 
the active population, there is a seed-bank of fixed size $M$ containing the dormant individuals. 

Given $M,\ N \in \N$, take $\epsilon \in [0,1]$ such that $\epsilon N \leq M$ and set $\delta 
= \epsilon N/M$. For convenience we assume that $\e N = \de M \in \N$. The dynamics of 
the model is as follows (see Fig.~\ref{fig:seedbank}):
\begin{itemize}
\item 
The $N$ active individuals produce $(1-\epsilon)N$ active individuals in the next generation, 
where every new individual randomly chooses a parent from the previous generation and 
adopts its type.
\item 
The remaining $\epsilon N (= \delta M)$ individuals from the active population produce 
individuals that become dormant, i.e., seeds in the next generation.
\item
From the seed-bank, $\delta M (= \epsilon N)$ individuals become active and leave the 
seed-bank. Therefore in the next generation the active population again consists of 
$(1-\epsilon) N + \epsilon N = N$ individuals, where the first term comes from the previously 
active population and the second term from the previously dormant population.
\item 
The remaining $(1-\delta) M$ seeds remain inactive and stay in the seed-bank. Therefore 
in the next generation the population in the seed-bank again consists of $(1-\delta) M 
+ \delta M = M$ individuals, where the first term comes from the previously dormant 
population and the second term from the previously active population.
\end{itemize}

\begin{figure}[htbp]
\begin{center}
\includegraphics[scale=0.4]{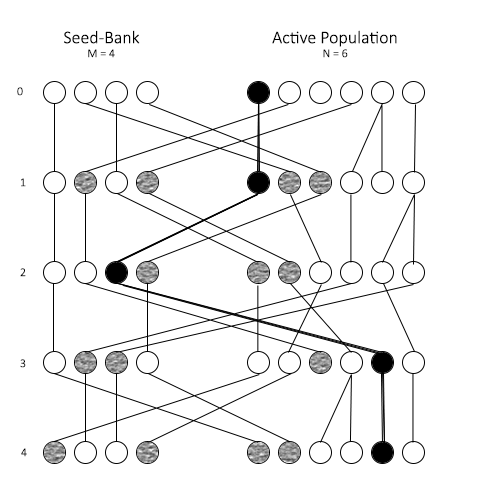}
\end{center}
\vspace{-.4cm}
\caption{\small Evolution over 4 generations of a population with $N+M=10$ individuals, of 
which $N=6$ are active and $M=4$ are inactive (in the seed-bank). The lines indicate
the genealogy of the population. The bold line shows a single line of ancestry.}
\label{fig:seedbank}
\end{figure}

\begin{Def}
\label{def:WFsb}
Let $M$, $N$, $\epsilon$, $\delta$ and $E$ be as above. Given an initial genetic type 
configuration $(\xi_0,\eta_0)$ with $\xi_0 \in E^N$, $\eta_0 \in E^M$, denote by
\begin{equation}
\xi_k = (\xi_k (i))_{i \in [N]},\ \ \eta_k = (\eta_k (j))_{j \in [M]}, \qquad  k \in \N_0, 
\end{equation}
the random genetic type configuration of, respectively, the active individuals and the 
dormant individuals in generation $k$. The discrete-time Markov chain $(\xi_k,\eta_k)_{k\in\N_0}$ 
with state space $E^N \times E^M$ is called the \textit{type configuration process} of 
the \textit{Wright-Fisher model with geometric seed-bank component}.
\end{Def}

\noindent
Note that the time a dormant individual stays in the seed-bank before becoming active is a 
random variable with a geometric distribution with parameter $\delta$, and that the times 
for different dormant individual are i.i.d. Similarly, the probability that an active individual 
becomes dormant is $\epsilon$. 

We want to understand the behaviour of the frequency of $a$ alleles in both the active 
population and the seed-bank. We therefore define
\begin{equation} 
X_k^N = \frac{1}{N} \sum_ {i \in [N]} \mathbbm{1}_{\{\xi_k(i)=a\}},
\qquad Y_k^M = \frac{1}{M} \sum_ {j \in [M]} \mathbbm{1}_{\{\eta_k(j)=a\}}, 
\end{equation}
which represent the fraction of individuals having type $a$ in generation $k$ in, respectively, 
the active population and the dormant population. Together they form a discrete-time Markov 
chain taking values in $I^N\times I^M$, with
\begin{equation} 
I^N = \left\{0,\frac{1}{N},\frac{2}{N},\ldots,1\right\} \subset [0,1],
\qquad I^M = \left\{0,\frac{1}{M},\frac{2}{M},\ldots,1\right\} \subset [0,1].
\end{equation}

Abbreviate
\begin{equation}
\mathbb{P}_{x,y} (\cdot) 
= \mathbb{P}(\:\cdot\: |\: X_0^N = x, Y_0^M = y), \qquad (x,y) \in I^N \times I^M.
\end{equation} 
The transition probabilities of this Markov chain can be characterised through the following 
proposition.

\begin{Prop}
Let $ c = \epsilon N = \delta M \in\N $. For $ (x,y), (\bar{x},\bar{y}) \in I^N \times I^M $,
\begin{align} 
&\P_{x,y} ( X_1^N=\bar{x},\ Y_1^M=\bar{y}) \nonumber\\
&\qquad = \sum_{i=0}^c \P_{x,y}(Z=i)\ \P_{x,y} (U=\bar{x}N-i)\ \P_{x,y} (V=\bar{y}M+i),
\end{align}
where $Z$, $U$, $V$ are independent under $\P_{x,y}$, with $Z$ $\cong \textnormal{Hyp}_{M,c,yM}$ 
(hypergeometric distribution), $U$ $\cong \textnormal{Bin}_{N-c,x}$ and $V$ $\cong 
\textnormal{Bin}_{c,x}$ (binomial distributions). 
\end{Prop}

\begin{proof}
The random variables have a simple interpretation:
\begin{itemize}
\item 
$Z$ is the number of individuals of type $a$ that become active in the next generation or, equivalently, 
the number of individuals at generation $1$ that are offspring of an individual of type $a$ at generation 
$0$.
\item 
$U$ is the number of individuals that are offspring of individuals of type $a$ in the previous generation 
(and therefore are themselves of type $a$).
\item 
$V$ is the number of individuals that are offspring of individuals of type $a$ in the previous generation.
\end{itemize}
With this interpretation the distributions of $Z$, $U$ and $V$ are immediate from Definition~\ref{def:WFsb}. 
By construction, $ X_1^N = \frac{U+Z}{N}$ and $Y_1^M = y+\frac{V-Z}{M}$, and so the claim follows. 
\end{proof}


\section{A multi-colony extension}
\label{sec:stepstone}

In this section we analyse a spatial version of the model introduced in Section~\ref{sec:seedbankdef}. 
Namely, we add \emph{migration}, i.e., we allow individuals to adopt the type of individuals in colonies 
different from their own (in the literature this is called the stepping stone model). In Section~\ref{mig}
we define the model without mutation, in Section~\ref{mut} we add mutation. In Section~\ref{pid} we 
turn to our key object of interest, the probability that two individuals drawn randomly from two given 
colonies are identical by descent. We derive a formula for this probability involving convolutions of 
matrices. In Section~\ref{Fpid} we simplify this formula by turning to Fourier analysis, which leads to 
our first main theorem: Theorem~\ref{MainTheorem}.


\subsection{Migration}
\label{mig}

Consider the \emph{discrete torus} $\T$ in any dimension $d\in\N$. This may be identified with the 
lattice $\Z^d \cap [0,L]^d$, $L\in\N$, with periodic boundaries. At each lattice site there is a colony, 
which consists of an active population of size $N$ and a seed-bank of size $M$. At each generation, 
every individual chooses a colony according to a random walk transition kernel $p(x,y)$, $x,y \in\T$, 
and chooses from that colony an individual that was active in that colony in the previous generation. 
We assume that $p(x,y)$, $x,y \in\T$, depends only on the distance between the two colonies $x$ 
and $y$, not on their position. This makes our system translation invariant. An example for a transition 
kernel that satisfies this assumption is
\begin{equation}
\label{transkernel} 
p(x,y) = (1-\nu)\delta_{x,y} + \nu q(x,y), \qquad x,y\in\T, 
\end{equation}
with $\delta_{x,y}=\mathbbm{1}_{\{x=y\}}$ and $q(x,y)$, $x,y\in\T$, a random walk transition kernel
that controls the migration. The parameter $\nu\in (0,1]$ is the migration probability. An example is 
the uniform nearest-neighbour model
\begin{equation} 
\label{qkernel}
q(0,z) = \begin{cases} 
\frac{1}{2d}, &\text{if } \|z\|=1, \\ 
0, &\text{otherwise}, 
\end{cases}
\end{equation}
where $\| \cdot \|$ is the lattice norm. This corresponds to a simple random walk on $\T$.

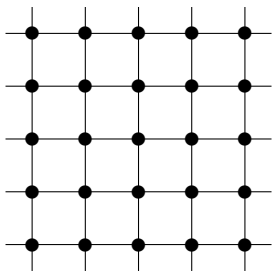
\begin{figure}[htbp]
\vspace{0.5cm}
\begin{center}
\setlength{\unitlength}{0.7cm}
\begin{picture}(12,5)(-3,0)
\put(1,0.5){\line(0,5){5}}
\put(2,0.5){\line(0,5){5}}
\put(3,0.5){\line(0,5){5}}
\put(4,0.5){\line(0,5){5}}
\put(5,0.5){\line(0,5){5}}
\put(0.5,1){\line(5,0){5}}
\put(0.5,2){\line(5,0){5}}
\put(0.5,3){\line(5,0){5}}
\put(0.5,4){\line(5,0){5}}
\put(0.5,5){\line(5,0){5}}
\put(1,1){\circle*{.25}}
\put(1,2){\circle*{.25}}
\put(1,3){\circle*{.25}}
\put(1,4){\circle*{.25}}
\put(1,5){\circle*{.25}}
\put(2,1){\circle*{.25}}
\put(2,2){\circle*{.25}}
\put(2,3){\circle*{.25}}
\put(2,4){\circle*{.25}}
\put(2,5){\circle*{.25}}
\put(3,1){\circle*{.25}}
\put(3,2){\circle*{.25}}
\put(3,3){\circle*{.25}}
\put(3,4){\circle*{.25}}
\put(3,5){\circle*{.25}}
\put(4,1){\circle*{.25}}
\put(4,2){\circle*{.25}}
\put(4,3){\circle*{.25}}
\put(4,4){\circle*{.25}}
\put(4,5){\circle*{.25}}
\put(5,1){\circle*{.25}}
\put(5,2){\circle*{.25}}
\put(5,3){\circle*{.25}}
\put(5,4){\circle*{.25}}
\put(5,5){\circle*{.25}}
\end{picture}
\end{center}
\vspace{-0.5cm}
\caption{\small A $5 \times 5$ torus, with edges connecting neighbouring vertices. The boundary is 
periodic: opposite ends horizontally and vertically are connected. Simple random walk corresponds 
to transitions north, east, south and west, each with probability $\tfrac14$.}
\label{fig-square}
\end{figure}

We are interested, for the system \emph{in equilibrium}, to compute the probability $\psi( (x,a),(y,b) )$ 
that two individuals drawn uniformly at random from two colonies $x,y\in\T$ in states $a,b\in\{0,1\}$  
are \textit{identical by descent}, i.e., their lineages coalesce. The states $a,b\in\{0,1\}$ indicate 
whether the individual is drawn from the dormant population (state 0) or the active population (state 1), 
e.g.\ $(x,0)$ means that an individual is drawn from the seed-bank in colony $x\in\T$.  
If $x=y$, then we require the two individuals to be distinct. We want to find an expression 
for the 4-vector
\begin{equation} 
\label{Psivector} 
\Psi_{x,y} = 
\begin{pmatrix}  
\psi( (x,0),(y,0) ) \\
\psi( (x,0),(y,1) ) \\
\psi( (x,1),(y,0) ) \\
\psi( (x,1),(y,1) ) 
\end{pmatrix}, \qquad x,y\in\T. 
\end{equation}
Note that, since $p(x,y)$ depends on $x-y$ only, the same is true for $\Psi_{x,y}$. Naturally, 
since $\T$ and $N$ are finite we have 
\begin{equation}  
\Psi_{x,y} = 
\begin{pmatrix}  
1 \\
1 \\
1 \\
1 
\end{pmatrix}, \qquad x,y\in\T. 
\end{equation}
Indeed, no matter what colonies the individuals are drawn from, their lineages move in 
and out of the seed-bank and migrate to other colonies in $\T$. Eventually, going infinitely 
far back in time, the lineages coalesce. The problem becomes more interesting when we 
modify the dynamics to include \emph{mutation}. 


\subsection{Migration and mutation}
\label{mut}

Assume that at each generation, every individual with probability $\mu\in (0,1)$ mutates 
to a \emph{new type} and with probability $1-\mu$ does not mutate and proceeds to
become either active, respectively, dormant (with probability $\de$, respectively, $\e$). In 
addition, if an individual is active and remains active, then it chooses an individual from 
any colony $y\in\T$ (with probability $(1-\e) p(x,y)$) and adopts its type. More precisely, 
a dormant individual in the seed-bank may
\begin{itemize}
\item 
mutate to a new type, with probability $\mu$,
\item 
maintain its type and remain in the seed-bank, with probability $(1-\mu)(1-\de)$,
\item 
maintain its type and become active, with probability $(1-\mu)\de$.
\end{itemize}
Similarly, an active individual may
\begin{itemize}
\item 
mutate to a new type, with probability $\mu$,
\item 
remain active and choose a random ancestor from colony $y\in\T$ (possibly $y=x$), 
with probability $(1-\mu)(1-\e) p(x,y)$,
\item 
maintain its type and become dormant, with probability $(1-\mu)\e$.
\end{itemize}
Since a lineage is interrupted when a mutation occurs, we define two individuals to be 
\textit{identical by descent} if their lineages coalesce \emph{before} a mutation affects 
either lineage. Our goal is to compute the $4$-vector $\Psi_{x,y}$ in \eqref{Psivector} 
in the presence of mutation, again for the system \emph{in equilibrium}. This will be 
achieved in several steps, resulting in Theorem~\ref{MainTheorem} below. Note that 
individuals may change their type multiple times. What matters for $\Psi_{x,y}$ is that 
the lineages of the two individuals drawn at $x$ and $y$ meet without encountering 
a mutation.


\subsection{Computation of the probability to be identical by descent}
\label{pid}

We begin by deriving a recursive relation for the family $\{\Psi_{x,y}\}_{x,y\in\T}$.

\begin{Prop} 
\label{multicolonyprop1}
For $x,y\in\T, $
\begin{equation}
\label{multicolonyrecursion}
\Psi_{x,y} = \Phi_{x,y} + A_{x,y} \Psi_{0,0} + \sum_{w,z\in\T} B_{w-x,z-y} \Psi_{w,z},
\end{equation}
where 
\begin{align}
\label{Phimatrix}
\Phi_{x,y} 
=&\ (1-\mu)^2 (1-\epsilon)^2 \begin{pmatrix} 0 \\ 0 \\ 0 \\  p^2(x,y) \frac{1}{N} \end{pmatrix}, \\[0.2cm]
\label{Amatrix} 
A_{x,y} =&\ - (1-\mu)^2 (1-\epsilon)^2 
\begin{pmatrix}
0 & 0 & 0 & 0 \\
0 & 0 & 0 & 0 \\
0 & 0 & 0 & 0 \\
0 & 0 & 0 & p^2(x,y) \frac{1}{N}
\end{pmatrix},\\
\label{bbtilderecursion}
B_{w-x,z-y} =&\ C \delta_{x,w} \delta_{y,z} + D_{w-x,z-y} ,
\end{align}
with
\begin{eqnarray}
&&C = (1-\mu)^2 \times \nonumber \\ 
&&\qquad
\begin{pmatrix}
(1-\delta)^2 & (1-\delta) \epsilon & (1-\delta) \epsilon & \epsilon^2 \\
\delta (1-\delta) & 0 & \delta \epsilon & 0 \\
\delta (1-\delta) & \delta \epsilon & 0 & 0 \\
\delta^2 & 0 & 0 & 0
\end{pmatrix}, 
\label{Cmatrix} \\[0.2cm]
&&D_{w-x,z-y} = (1-\mu)^2 (1-\epsilon) \times \nonumber \\
&&\qquad
\begin{pmatrix}
0 & 0                          		& 0                          			& 0 \\
0 & (1-\delta)p(y,z)\delta_{x,w}  & 0                          			& \epsilon p(y,z) \de_{x,w}\\
0 & 0                          		& (1-\delta)p(x,w)\delta_{y,z} 	&  \epsilon p(x,w) \de_{y,z}\\
0 & \delta p(y,z)\delta_{x,	w}    	& \delta p(x,w)\delta_{y,z}          & (1-\epsilon) p(x,w) p(y,z)
\end{pmatrix}. 
\label{Dmatrix}
\end{eqnarray}
\end{Prop}

\begin{proof}
We begin by writing down a recursion relation for 
\begin{equation}
\Psi^{(n)}_{x,y} = 
\begin{pmatrix}  
\psi_n( (x,0),(y,0) ) \\
\psi_n( (x,0),(y,1) ) \\
\psi_n( (x,1),(y,0) ) \\
\psi_n( (x,1),(y,1) ) 
\end{pmatrix}, \qquad x,y\in\T,
\end{equation} 
the probability at time $n$ that two individuals randomly drawn from colonies $x,y$ 
are identical by descent. This reads as follows:
\begin{align}
\label{psinrec} 
&\psi_{n+1}( (x,0),(y,0) ) = (1-\mu)^2 \Big[ (1-\delta)^2 \psi_n( (x,0),(y,0) ) 
+ \epsilon^2 \psi_n( (x,1),(y,1) ) \nonumber\\
&\qquad + \epsilon (1-\delta) \left[ \psi_n( (x,1),(y,0) ) 
+ \psi_n( (x,0),(y,1) )\right] \Big],
\nonumber\\
&\psi_{n+1}( (x,1),(y,0) ) = (1-\mu)^2 \bigg[ \delta (1-\delta) \psi_n( (x,0),(y,0) ) 
+ \delta \epsilon \psi_n( (x,0),(y,1) )\nonumber \nonumber\\
&\left.\qquad + \sum_{w\in\T} (1-\epsilon) (1-\delta) p(x,w) \psi_n( (w,1),(y,0) )
+ \epsilon (1-\epsilon) p(x,w) \psi_n( (w,1),(y,1) )\right],
\nonumber\\
&\psi_{n+1}( (x,0),(y,1) ) = (1-\mu)^2 \bigg[ \delta (1-\delta) \psi_n( (x,0),(y,0) ) 
+ \delta \epsilon \psi_n( (x,1),(y,0) )\nonumber \nonumber\\
&\left. \qquad+ \sum_{z\in\T} (1-\epsilon) (1-\delta) p(y,z) \psi_n( (x,0),(z,1) ) 
+ \epsilon (1-\epsilon) p(y,z) \psi_n( (x,1),(z,1) )\right],
\nonumber\\
&\psi_{n+1}( (x,1),(y,1) ) = (1-\mu)^2 \bigg[ \delta^2 \psi_n( (x,0),(y,0) ) \nonumber\\
&\left.\qquad + \sum_{z\in\T} \delta(1-\epsilon) p(y,z) \psi_n( (x,0),(z,1) ) \right.\nonumber\\
&\left.\qquad + \sum_{w\in\T} \delta(1-\epsilon) p(x,w) \psi_n( (w,1),(y,0) )\right.\nonumber \\
&\left.\qquad + \sum_{\substack{w,z\in\T\\ w\neq z}} (1-\epsilon)^2 p(x,w) p(y,z) 
\psi_n( (w,1),(z,1) \nonumber)\right. \\
&\left.\qquad + \sum_{z\in\T} (1-\epsilon)^2 p(x,z) p(y,z) \left[ \frac{1}{N} 
+ \left(1-\frac{1}{N}\right) \psi_n( (z,1),(z,1) ) \right] \right].
\end{align}
The reasoning behind the above expressions comes from considering the possible choices 
of the two individuals. In the first expression, for example, the drawn individuals are both in 
the seed-bank and there are three scenarios: 
\begin{itemize}
\item 
they were both in the seed bank in the previous generation (with probability $1-\delta$ 
each, independently of each other),
\item 
they are both offspring of active individuals (with probability $\epsilon$ each),
\item 
one was in the seed-bank in the previous generation and did not become active (with 
probability $1-\delta$), and the other is offspring of an active individual (with probability 
$\epsilon$).
\end{itemize}
Note that if individuals are in the seed-bank, then the individuals they are offsping of cannot 
be from a different colony, which is why the transition kernels do not always appear in the 
recursive relations. When they appear, they are always multiplied by $1-\epsilon$, since the 
transition kernels are probabilities \textit{conditional} on the event that the individuals were 
not in the seed-bank in the previous generation. In the last term of $\psi_{n+1}( (x,1),(y,1) )$ 
we are looking at the event in which the ancestors of the two individuals are in the same colony: 
either they are the same individual (with probability $1/N$ and the iteration ends) or they are 
two distinct individuals (with probability $1-\frac{1}{N}$). In this equation we add and subtract 
the term $(1-\mu)^2 (1-\epsilon)^2 p(x,w) p(y,z) \psi_n( (w,1),(z,1) )$ for $w=z$ (the expression in 
the fourth line) to obtain
\begin{align}
&\psi_{n+1}( (x,1),(y,1) ) = (1-\mu)^2 \bigg[ \delta^2 \psi_n( (x,0),(y,0) ) \nonumber\\
&\left. \qquad + \sum_{z\in\T} \delta(1-\epsilon) p(y,z) \psi_n( (x,0),(z,1) ) \right.\nonumber\\
&\left. \qquad + \sum_{w\in\T} \delta(1-\epsilon) p(x,w) \psi_n( (w,1),(y,0) )\right.\nonumber \\
&\left. \qquad + \sum_{w,z\in\T} (1-\epsilon)^2 p(x,w) p(y,z) \psi_n( (w,1),(z,1) \nonumber)\right. \\
&\left. \qquad + \sum_{z\in\T} (1-\epsilon)^2 p(x,z) p(y,z) \left[ \frac{1}{N} 
- \frac{1}{N} \psi_n( (z,1),(z,1) ) \right] \right].
\end{align}
By ergodicity, we have
\begin{equation}
 \lim_{n\to\infty} \Psi_{x,y}^{(n)} = \Psi_{x,y}\ \ \ \forall x,y\in\T.
\end{equation}
Therefore, in equilibrium, we may drop the time indices from the above expressions and obtain
\begin{equation}
\begin{array}{ll}
\begin{pmatrix}  
\psi( (x,0),(y,0) ) \\
\psi( (x,0),(y,1) ) \\
\psi( (x,1),(y,0) ) \\
\psi( (x,1),(y,1) ) 
\end{pmatrix}
\\ 
&\\
&\hspace{-14cm}
= (1-\mu)^2 
\begin{pmatrix}
(1-\delta)^2 & (1-\delta) \epsilon & (1-\delta) \epsilon & \epsilon^2 \\
\delta (1-\delta) & 0 & \delta \epsilon & 0 \\
\delta (1-\delta) & \delta \epsilon & 0 & 0 \\
\delta^2 & 0 & 0 & 0
\end{pmatrix}
\begin{pmatrix}  
\psi( (x,0),(y,0) ) \\
\psi( (x,0),(y,1) ) \\
\psi( (x,1),(y,0) ) \\
\psi( (x,1),(y,1) ) 
\end{pmatrix} 
\\
&\\
&\hspace{-14cm}
+\sum_{w,z\in\T} (1-\mu)^2 (1-\epsilon) \times\\
\begin{pmatrix}
0 & 0                          		& 0                          			& 0 \\
0 & (1-\delta)p(y,z)\delta_{x,w}  & 0                          			& \epsilon p(y,z) \de_{x,w} \\
0 & 0                          		& (1-\delta)p(x,w)\delta_{y,z} 	&  \epsilon p(x,w) \de_{y,z}\\
0 & \delta p(y,z)\delta_{x,	w}    	& \delta p(x,w)\delta_{y,z}          & (1-\epsilon) p(x,w) p(y,z)
\end{pmatrix} 
\begin{pmatrix}  
\psi( (w,0),(z,0) ) \\
\psi( (w,0),(z,1) ) \\
\psi( (w,1),(z,0) ) \\
\psi( (w,1),(z,1) ) 
\end{pmatrix} 
\\
&\\
&\hspace{-14cm}
- \sum_{z\in\T} (1-\mu)^2 (1-\epsilon)^2 \times\\
\begin{pmatrix}
0 & 0 & 0 & 0 \\
0 & 0 & 0 & 0 \\
0 & 0 & 0 & 0 \\
0 & 0 & 0 &  p(x,z)p(y,z)\frac{1}{N}
\end{pmatrix} 
\begin{pmatrix}  
\psi( (z,0),(z,0) ) \\
\psi( (z,0),(z,1) ) \\
\psi( (z,1),(z,0) ) \\
\psi( (z,1),(z,1) ) 
\end{pmatrix} 
\\
&\\
&\hspace{-14cm}
+ \sum_{z\in\T}  (1-\mu)^2 (1-\e)^2 
\begin{pmatrix}
0 \\ 0 \\ 0 \\  p(x,z)p(y,z) \frac{1}{N}
\end{pmatrix}.
\end{array}
\end{equation}
After using that $\sum_{z\in\T} p(x,z)p(y,z) = \sum_{z\in\T} p(x,z)p(z,y) = p^2(x,y)$, and 
$\Psi_{z,z} = \Psi_{0,0}$ by translation invariance, we obtain the expression in 
\eqref{multicolonyrecursion}.
\end{proof}

Now that we have a recursive relation for $\Psi_{x,y}$ as expressed in Proposition~\ref{multicolonyprop1}, 
we proceed to solve this relation to find a closed form expression for $\Psi_{x,y}$. To do so, we iterate 
\eqref{multicolonyrecursion} and, after noting that the last summand tends to 0 as the number of iterations 
tends to infinity (since $0<\mu< 1$), we obtain the following expression.

\begin{Prop}
For all $x,y\in\T$,
\begin{equation}
\label{mastereq}
\Psi_{x,y} = \frac{1-\Psi_{0,0}^{(4)}}{N} \sum_{n\in\N_0}  (B\ast\Gamma)^{(n)}_{x,y},
\end{equation}
where $\Psi_{0,0}^{(4)}$ is the $4$-th entry of the $4$-vector $\Psi_{0,0}$ and (with $\mathds{1}$ the 
$(4 \times 4)$-identity matrix)
\begin{align}
(B \ast \Gamma)^{(n)}_{x,y} =&\ \sum_{w,z\in\T} B^{(n)}_{w-x,z-y} \Gamma_{w,z}, \\
\label{Brec} B^{(n)}_{w-x,z-y} =&\ \sum_{w',z'\in\T} B^{(n-1)}_{w'-x,z'-y} B_{w-w',z-z'}, \\
\label{Bzero} B^{(0)}_{w-x,z-y} =&\ \mathds{1} \delta_{x,w} \delta_{y,z}, \\
\label{Gamma} \Gamma_{w,z} =&\ (1-\mu)^2 (1-\e)^2 
\begin{pmatrix} 0 \\ 0 \\ 0 \\  p^2(w,z) \end{pmatrix}.
\end{align}
\end{Prop}

\begin{proof} 
The claim follows by repeatedly substituting $\Psi_{w,z}$ into \eqref{multicolonyrecursion}. 
Indeed,
\begin{align}
\Psi_{x,y} =&\: \Phi_{x,y} + A_{x,y} \Psi_{0,0} 
+ \sum_{w,z\in\T} B_{w-x,z-y} \Psi_{w,z} \nonumber\\
=&\: \Phi_{x,y} + D_{x,y} \Psi_{0,0} 
+ \sum_{w,z\in\T} B_{w-x,z-y} \left( \Phi_{w,z} + A_{w,z} \Psi_{0,0} 
+ \sum_{w',z'\in\T} B_{w'-z,z'-z} \Psi_{w',z'} \right) \nonumber\\
=&\: \left( \mathds{1}\delta_{xw}\delta_{yz} + \sum_{w,z\in\T} B_{w-x,z-y} \right) \Phi_{w,z} 
+ \left( \mathds{1}\delta_{xw}\delta_{yz} + \sum_{w,z\in\T} B_{w-x,z-y} \right) A_{w,z} \Psi_{0,0} 
\nonumber \\
& + \sum_{w,z\in\T} \sum_{w',z'\in\T} B_{w-x,z-y} B_{w'-z,z'-z} \Psi_{w',z'} \nonumber\\
=&\: \sum_{k=0}^1 \sum_{w,z\in\T} B^{(k)}_{w-x,z-y}  \left( \Phi_{w,z} 
+ A_{w,z} \Psi_{0,0} \right) \nonumber \\
& +   \sum_{w'',z''\in\T}  B^{(2)}_{w''-x,z''-y} \Psi_{w'',z''},
\end{align}
where $B^{(0)}_{w-x,z-y}$ is defined by \eqref{Bzero} and $B^{(n)}_{w-x,z-y}$ is defined 
recursively as in \eqref{Brec}. After $n$ substitutions we obtain
\begin{equation}
\Psi_{x,y} = \sum_{k=0}^n \sum_{w,z\in\T} B^{(k)}_{w-x,z-y} \left( \Phi_{w,z} 
+ A_{w,z} \Psi_{0,0} \right) +   \sum_{w'',z''\in\T}  B^{(n+1)}_{w''-x,z''-y} \Psi_{w'',z''}.
\end{equation}
Letting $n\to\infty$ and noting that $\lim_{n\to\infty} B^{(n)}_{w-x,z-y} = 0$ (each term 
is finite and is multiplied by $(1-\mu)^{2n}$ with $0<\mu<1$), we see that the summand in 
the second line of the last equality tends to 0. We therefore obtain
\begin{align}
\label{mastereqpsi0}
\Psi_{x,y} =&\ \sum_{k\in\N_0} \sum_{w,z\in\T} B^{(k)}_{w-x,z-y} 
\left( \Phi_{w,z} + A_{w,z} \Psi_{0,0} \right). 
\end{align}
We can rewrite
\begin{align}
\Phi_{w,z} + A_{w,z} \Psi_{0,0} = (1-\mu)^2 (1-\e)^2 
\begin{pmatrix} 0 \\ 0 \\ 0 \\ \frac{1-\Psi_{0,0}^{(4)}}{N} p^2(w,z) \end{pmatrix} 
= \frac{1-\Psi_{0,0}^{(4)}}{N}\, \Gamma_{w,z}
\end{align}
with $\Gamma_{w,z}$ given by \eqref{Gamma}, and the claim follows.
\end{proof}

Taking $x=y=0$ in \eqref{mastereqpsi0}, we get
\begin{equation}
\Psi_{0,0} = \left(\mathds{1}- \sum_{n\in\N_0}  (B\ast A)^{(n)}_{0,0}\right)^{-1 } 
\left(\sum_{n\in\N_0} (B\ast\Phi)^{(n)}_{0,0}\right).
\end{equation}
This can be substituted into \eqref{mastereq} to obtain an explicit expression for $\Psi_{x,y}$. 
Since this expression contains convolutions, we turn to Fourier analysis to gain more insight 
into the properties of $\Psi_{x,y}$.


\subsection{Fourier analysis}
\label{Fpid}

Let $\hat{\T} = \{ 0,\frac{1}{L}, \ldots, \frac{L-1}{L}\}^d$. Define for $f\colon\,\T\times\T \mapsto \R$, 
$x,y \in \T$ and $\theta,\eta \in \hat{\T}$,
\begin{align}
\label{discFourier}
\hat{f}(\theta,\eta) 
=&\ \sum_{x,y\in \T} f_{x,y} \: e^{2\pi i ( x\cdot\theta + y\cdot \eta)}, \\
\label{discFourierinv}
f_{x,y} =&\ \frac{1}{|\hat{\T}|^2} \sum_{\theta,\eta \in \hat{\T}} \hat{f}(\theta,\eta) \: 
e^{-2\pi i (x\cdot \theta + y\cdot \eta)}.
\end{align}

\begin{Prop}
For $\theta,\eta \in \hat{\T}$,
\begin{equation}
\label{fourierPsi}
\hat{\Psi}(\theta,\eta) =  \frac{1-\Psi_{0,0}^{(4)}}{N} 
\big( \mathds{1} -  \hat{B}(\theta,\eta)  \big)^{-1} \hat{\Gamma}(\theta,\eta).
\end{equation}
\end{Prop}

\begin{proof}
By the linearity of the Fourier transform and the convolution theorem, we get from 
\eqref{mastereq} that
\begin{equation}
\label{fouriermastereq}
\hat{\Psi}(\theta,\eta) =  \frac{1-\Psi_{0,0}^{(4)}}{N} 
\sum_{n\in\N_0} \hat{B}^{(n)}(\theta,\eta)\, \hat{\Gamma}(\theta,\eta).
\end{equation}
Since $B^{(n)}_{x-w,y-z}$ is defined by the recursion in \eqref{Brec}, we have
\begin{equation}
\hat{B}^{(n)}(\theta,\eta) = \hat{B}^{(n-1)}(\theta,\eta) \hat{B}(\theta,\eta) 
=  \big( \hat{B}(\theta,\eta) \big)^n,
\end{equation}
and therefore
\begin{equation}
\label{fouriergamma}
\sum_{n\in\N_0} \hat{B}^{(n)}(\theta,\eta) \hat{\Gamma}(\theta,\eta)  
= \sum_{n\in\N_0} \big( \hat{B}(\theta,\eta) \big)^n \hat{\Gamma}(\theta,\eta) 
= \big( \mathds{1} - \hat{B}(\theta,\eta)  \big)^{-1} \hat{\Gamma}(\theta,\eta).
\end{equation}
The claim follows after substitution of \eqref{fouriergamma} into \eqref{fouriermastereq}.
\end{proof}

Our next objective is to compute the right-hand side of \eqref{fourierPsi}. We obtain
\begin{align}
\label{Bhatrep}
\hat{B}(\theta,\eta) = \sum_{u,v\in\T} B_{u,v} \: e^{2\pi i (\theta\cdot u + \eta \cdot v)} 
= \sum_{u,v\in\T} \left( C \: \de_{u,0} \: \de_{v,0} + D_{u,v} \right) 
e^{2\pi i (\theta\cdot u + \eta \cdot v)} = C + \hat{D}(\theta,\eta),
\end{align}
where $C$ is given in \eqref{Cmatrix} and
\begin{equation}
\label{Dhatrep}
\hat{D}(\theta,\eta) = (1-\mu)^2 (1-\e) 
\begin{pmatrix} 
0	& 0				& 0				& 0 \\
0	& (1-\de) \hat{p}(\eta)	& 0				& \e \hat{p}(\eta) \\
0 	& 0				& (1-\de) \hat{p}(\theta)	& \e \hat{p}(\theta) \\
0	& \de \hat{p}(\eta)		& \de \hat{p}(\theta)	& (1-\e) \hat{p}(\theta) \hat{p}(\eta)
\end{pmatrix},
\end{equation}
with
\begin{equation}
\hat{p}(\theta) = \sum_{z\in\T} p(0,z) e^{2\pi i \theta\cdot z}.
\end{equation}
Computing $( \mathds{1} -  \hat{B}(\theta,\eta) )^{-1}$, we find
\begin{equation}
\label{oneminusbinverse}
\big( \mathds{1} -  \hat{B}(\theta,\eta)  \big)^{-1} = \frac{1}{r_0(\theta,\eta)} 
\begin{pmatrix}
r_{1,1}(\theta,\eta) & r_{1,2}(\theta,\eta) & r_{1,3}(\theta,\eta) & r_{1,4}(\theta,\eta) \\
r_{2,1}(\theta,\eta) & r_{2,2}(\theta,\eta) & r_{2,3}(\theta,\eta) & r_{2,4}(\theta,\eta) \\
r_{3,1}(\theta,\eta) & r_{3,2}(\theta,\eta) & r_{3,3}(\theta,\eta) & r_{3,4}(\theta,\eta) \\
r_{4,1}(\theta,\eta) & r_{4,2}(\theta,\eta) & r_{4,3}(\theta,\eta) & r_{4,4}(\theta,\eta)
\end{pmatrix},
\end{equation}
where the 17 functions in the above expression are polynomials of degree $\le 4$ in $\hat{p}(\theta)$ 
and $\hat{p}(\eta)$ whose coefficients depend on the parameters $\delta,\e$. Since 
$( \mathds{1} - \hat{B}(\theta,\eta))^{-1}$ pre-multiplies $\hat{\Gamma}(\theta,\eta)$, whose first 
three entries are $0$, we only need the entries in the fourth column. These are given by the 
following, where we abbreviate $m=(1-\mu)^2$:
\begin{align}
\label{oneminusbhat}
r_0(\theta,\eta) 
= &\ (1-m^2\de^2\e^2)\Big( 1-m \Big( 1-\de \big( 2-2\e-\de (1-m\e^2) \big) \Big) \Big) \nonumber \\
&\ - m^2(1-\e)^2 \hat{p}(\eta)^2 \big( \de\e - (1-\de)(1-\e)\hat{p}(\theta) \big) \nonumber \\ 
&\qquad\times\Big( 1-m(1-\de(2-\de-\e)) 
-m\big(1-m(1-\de)^2\big) (1-\de)(1-\e) \hat{p}(\theta) \Big) \nonumber \\
&\ -m(1-\e)(1-\de) \bigg( 1-m \Big(1-\de\big( 2-\e+\de (1-m\e^2 (1-m\de\e)) \big)\Big) \bigg) 
\big( \hat{p}(\eta)+\hat{p}(\theta) \big) \nonumber \\
&\ +m^2\de\e(1-\e)^2 \big( 1-m \big( 1-\de(2-\de-\e) \big) \big) \hat{p}(\theta)^2 \nonumber \\
&\ -m(1-\e)^2 \hat{p}(\eta) \hat{p}(\theta) \Big( \big(1-m(1-\de)^2\big)^2 
-2m\de\e+m^2\de^2\e^2 \nonumber\\
&\qquad -m(1-\de)(1-\e) \big( 1-m(1-\de)^2 (1+m\de\e)\big) \hat{p}(\theta) \Big)
\end{align}
and
\begin{align}
r_{1,4}(\theta,\eta) = &\ m\e^2 \big( (1-m\de\e)^2 - m^2 (1-\de)^2 (1-\e)^2 \hat{p}(\theta) 
\hat{p}(\eta) \big), \nonumber \\
r_{2,4}(\theta,\eta) = &\ m\e \Big( m\de\e(1-\de)(1-m\de\e) + (1-\e) \hat{p}(\eta) 
\Big( 1-m(1-\de(2-\de-\e)) \nonumber \\ 
&\qquad -m(1-m(1-\de)^2)(1-\de)(1-\e) \hat{p}(\theta) \Big) \Big), \nonumber \\
r_{3,4}(\theta,\eta) = &\ m\e \Big( m\de\e(1-\de)(1-m\de\e) + (1-\e) \hat{p}(\theta) 
\Big( 1-m(1-\de(2-\de-\e)) \nonumber \\ 
&\qquad -m(1-m(1-\de)^2)(1-\de)(1-\e) \hat{p}(\eta) \Big) \Big), \nonumber \\
r_{4,4}(\theta,\eta) = &\ -m^2\de\e(1-\de)^2 
\big( 1-m\de\e-(1-\de)(1-\e)\hat{p}(\eta) \big)  \nonumber \\
&\qquad +\Big( 1-m(1-\de(2-\de-\e)) -m\big(1-m(1-\de)^2\big)(1-\de)(1-\e)\hat{p}(\eta) \Big) \nonumber \\
&\qquad\qquad \times\Big( 1-m\de\e-m(1-\de)(1-\e)\hat{p}(\theta) \Big).
\label{rindices}
\end{align}

Looking at $\hat{\Gamma}(\theta,\eta)$, we have
\begin{align}
\label{gammahattheta}
\hat{\Gamma}(\theta,\eta)
=& \ (1-\mu)^2 \sum_{w,z\in \T} \Gamma_{w,z} e^{2\pi i (\theta \cdot w + \eta \cdot z)}\nonumber\\ 
=&\ (1-\mu)^2 \sum_{w,z\in \T} \begin{pmatrix} 0 \\ 0 \\ 0 \\ (1-\e)^2 \frac{1}{N} p^2(w,z) \end{pmatrix} 
e^{2\pi i (\theta \cdot w + \eta \cdot z) } \nonumber\\
=&\ \frac{1}{N} (1-\mu)^2 (1-\e)^2 
\begin{pmatrix} 0 \\ 0 \\ 0 \\   \hat{p}(\theta) \hat{p}(\eta)  
\end{pmatrix}  \delta_{\theta,-\eta}.
\end{align}
Therefore we get 
\begin{align}
\label{onembhatphihat}
\big( \mathds{1} -  \hat{B}(\theta,\eta)  \big)^{-1} \hat{\Gamma}(\theta,\eta)  
=&\ \frac{(1-\mu)^2 (1-\e)^2}{N} \frac{1}{r_0(\theta,\eta)} \begin{pmatrix}  
r_{1,4}(\theta,\eta) \\
r_{2,4}(\theta,\eta) \\
r_{3,4}(\theta,\eta) \\
r_{4,4}(\theta,\eta) 
\end{pmatrix} \hat{p}(\theta) \hat{p}(\eta)  \delta_{\theta,-\eta},
\end{align}
Note that, because of the multiplicative factor $\de_{\theta,-\eta}$, $\hat{\Psi}(\theta,\eta)$ 
in \eqref{fourierPsi} actually depends on $\theta$ only. This is in agreement with the fact 
that $\Psi_{x,y}$ actually depends on $x-y$ only.  Henceforth we write $\hat{\Psi}(\theta,-\theta)
=\hat{\Psi}(\theta)$, and similarly for the other symbols.

At this point, all the terms in \eqref{fourierPsi} are known objects, except for $\Psi_{0,0}$. In order 
to compute $\Psi_{0,0}$, we take the Fourier transform of \eqref{mastereqpsi0} to obtain
\begin{equation}
\label{psihatforpsi00}
\hat{\Psi}(\theta) =  (\mathds{1} - \hat{B}(\theta))^{-1} 
\big( \hat{\Phi}(\theta) + \hat{A}(\theta) \Psi_{0,0}\big).
\end{equation}
We then use the Fourier inversion formula \eqref{discFourierinv}, which gives
\begin{equation}
\frac{1}{|\hat{\T}|} \sum_{\theta \in \hat{\T}} \hat{\Psi}(\theta)  = \Psi_{0,0}.
\end{equation}
Substitution of \eqref{psihatforpsi00} therefore yields
\begin{equation}
\label{psi00}
\Psi_{0,0} = \left( \mathds{1} - \frac{1}{|\hat{\T}|} \sum_{\theta \in\hat{\T}}  (\mathds{1} 
- \hat{B}(\theta))^{-1} \hat{A}(\theta)\right)^{-1} 
\left(\frac{1}{|\hat{\T}|} \sum_{\theta \in \hat{\T}}  (\mathds{1} 
- \hat{B}(\theta))^{-1} \hat{\Phi}(\theta)\right).
\end{equation}
 
With \eqref{onembhatphihat} and \eqref{psi00} we have obtained an explicit expression 
for $\hat{\Psi}(\theta)$ in \eqref{fourierPsi}, which we summarise in a theorem.
Abbreviate
\begin{equation}
K = \frac{(1-\mu)^2 (1-\e)^2}{N},
\end{equation}
let
\begin{equation}
\label{si4def}
s_{i,4}(\theta) = K\, \hat{p}(\theta)^2\,\frac{r_{i,4}(\theta)}{r_0(\theta)}, \qquad i=1,2,3,4,
\end{equation}
where $r_0(\theta)$ and $r_{i,4}(\theta)=r_{i,4}(\theta)$, $i=1,2,3,4$, are defined in 
\eqref{oneminusbhat} and \eqref{rindices}, and let
\begin{equation}
s_{i,4} = \frac{1}{|\hat{\T}|} \sum_{\theta \in \hat{\T}} s_{i,4}(\theta), \qquad i=1,2,3,4.
\end{equation}. 

\begin{Thm}
\label{MainTheorem}
For every $\theta \in \hat{\T}$,
\begin{equation}
\label{fourierPsialt}
\hat{\Psi}(\theta) = \sum_{x \in \T} \Psi_{0,x}\,e^{2\pi i(x \cdot \theta)}
=  \frac{1-\Psi_{0,0}^{(4)}}{N}\, 
\begin{pmatrix}  
s_{1,4}(\theta)\\
s_{2,4}(\theta) \\
s_{3,4}(\theta) \\
s_{4,4}(\theta)
\end{pmatrix}
\end{equation}
with
\begin{equation}
\label{Psifouriden}
\Psi_{0,0}^{(4)} = \frac{ s_{4,4} }{ 1 - s_{4,4} }.
\end{equation}
\end{Thm}

\begin{proof}
The vector $\hat{\Phi}(\theta)$ and the matrix $\hat{A}(\theta)$ are simple, namely, 
\eqref{Phimatrix} and \eqref{Amatrix} give
\begin{equation}
\hat{\Phi}(\theta)
= K\,\hat{p}(\theta)^2
\begin{pmatrix} 0 \\ 0 \\ 0 \\ 1
\end{pmatrix}, 
\quad
\hat{A}(\theta) = - K\,\hat{p}(\theta)^2
\begin{pmatrix}
0 & 0 & 0 & 0 \\
0 & 0 & 0 & 0 \\
0 & 0 & 0 & 0 \\
0 & 0 & 0 & 1
\end{pmatrix}.
\end{equation}
Hence \eqref{oneminusbinverse} gives \eqref{fourierPsialt} with
\begin{equation}
\Psi_{0,0} = 
\left( \mathds{1} - 
\begin{pmatrix}
0 & 0 & 0 & s_{1,4} \\
0 & 0 & 0 & s_{2,4} \\
0 & 0 & 0 & s_{3,4} \\
0 & 0 & 0 & s_{4,4}
\end{pmatrix}
\right)^{-1}
\begin{pmatrix}  
s_{1,4} \\
s_{2,4} \\
s_{3,4} \\
s_{4,4}
\end{pmatrix}
= \left(\mathds{1} + \frac{1}{1-s_{4,4}}
\begin{pmatrix}
0 & 0 & 0 & s_{1,4} \\
0 & 0 & 0 & s_{2,4} \\
0 & 0 & 0 & s_{3,4} \\
0 & 0 & 0 & s_{4,4}
\end{pmatrix}
\right)
\begin{pmatrix}  
s_{1,4} \\
s_{2,4} \\
s_{3,4} \\
s_{4,4}
\end{pmatrix}.
\end{equation}
The latter in turn yields \eqref{Psifouriden}.
\end{proof}

The formula in \eqref{fourierPsialt} is difficult to analyse. The $\theta$-dependence sits in 
the quotients of $r_{i,4}(\theta)$, $i=1,2,3,4$, and $r_0(\theta)$. These are second-degree 
and fourth-degree polynomials in $\hat{p}(\theta)$, respectively, with coefficients that depend 
on the parameters $\delta$, $\epsilon$ and $\mu$. In order to find $\Psi_{0,x}$ we need to Fourier 
invert the $4$-vector in the right-hand side of \eqref{fourierPsialt}.

\begin{Remark}
{\rm The result in Theorem~\ref{MainTheorem} is \emph{general}. Our colonies form a discrete torus 
$\T$, but we could arrange them on any regular lattice, since only translation invariance and 
periodicity are needed. The formulas are also valid for a generic random walk transition kernel 
$p(x,y)$,\ $x,y\in\T$, beyond the special case considered in \eqref{transkernel}--\eqref{qkernel}. 
Moreover, if we would have had $K\in\N$ seed-banks instead of one, then we would have to work 
with $(K+1)^2 \times (K+1)^2$ matrices rather than $4 \times 4$ matrices, but the structure would 
be the same.}
\end{Remark}


\section{Special choice of parameters}
\label{spec-choice-par}

In Section~\ref{sws} we look at the special case $M=N$, for which $\de = \e \frac{M}{N} 
= \e$, and that $0 < \de \ll 1$. We refer to this case as the \emph{symmetric slow seed-bank}
( see Fig.~\ref{fig-sssb}). This choice will allow us to simplify the polynomials $r_{i,4}(\theta)$, 
$i=1,2,3,4$, and $r_0(\theta)$ appearing in \eqref{si4def}, and deduce from 
Theorem~\ref{MainTheorem} a more manageable formula for $\Psi_{0,x}$, $x\in\T$, stated 
in Theorems~\ref{thm:smalld} and \ref{thm:smalldalt}, in terms of the Green function associated 
with the random walk $q$ in \eqref{transkernel}. In Section~\ref{gf} we recall what is known 
for this Green function, both on the infinite torus and the finite torus. In Section~\ref{wmut} 
we use this information to obtain explicit scaling expressions when $0 < \mu/\nu \ll 1$ 
(= slower mutation than migration). In Sections~\ref{reg}--\ref{regalt} we use these expressions 
to discuss various regimes for $\Psi_{0,x}$, $x\in\T$, as a function of $N$, $L$ and $\mu/\nu$, 
stated in Theorems~\ref{d0exp1}--\ref{d>0exp2}.

\begin{figure}[htbp]
\vspace{-0.5cm}
\begin{center}
\setlength{\unitlength}{0.7cm}
\begin{picture}(10,5)(-3,-1)
\put(-4,0){\line(1,0){10}}
\put(1,-0.25){\line(0,1){0.5}}
\put(-2,0.3){$N$}
\put(3.5,0.3){$N$}
\put(0,1){\vector(1,0){2}}
\put(2,2){\vector(-1,0){2}}
\put(0.9,1.2){$\delta$}
\put(0.9,2.2){$\delta$}
\put(-2.3,-.5){\tiny\mbox{active}}
\put(3,-.5){\tiny\mbox{dormant}}
\end{picture}
\end{center}
\vspace{-0.5cm}
\caption{\small Symmetric slow seed-bank: equal size $N$ for the active and the dormant 
population, with equal and small crossover rate $0<\delta \ll 1$ in both directions.}
\label{fig-sssb}
\end{figure}
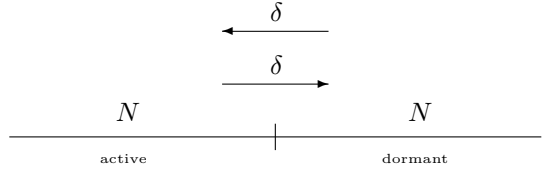


\subsection{Symmetric slow seed-bank}
\label{sws}

Let
\begin{equation}
\label{abcF}
\hat{\alpha}(\theta) = \frac{m\hat{p}(\theta)^2}{1-m\hat{p}(\theta)^2}, \quad
\hat{\beta}(\theta) = \frac{m^2\hat{p}(\theta)^3}{(1-m\hat{p}(\theta))(1-m\hat{p}(\theta)^2)}, \quad
\hat{\gamma}(\theta) = \frac{m\hat{p}(\theta)^2}{(1-m\hat{p}(\theta)^2)^2},
\end{equation}
and
\begin{equation}
\label{abcFinv}
\alpha(x) = \frac{1}{|\hat{\T}|} \sum_{\theta\in\hat{\T}} \hat{\alpha}(\theta)\,e^{-2\pi i(\theta\cdot x)}, \quad
\beta(x) = \frac{1}{|\hat{\T}|} \sum_{\theta\in\hat{\T}} \hat{\beta}(\theta)\,e^{-2\pi i(\theta\cdot x)}, \quad
\gamma(x) = \frac{1}{|\hat{\T}|} \sum_{\theta\in\hat{\T}} \hat{\gamma}(\theta)\,e^{-2\pi i(\theta\cdot x)}.
\end{equation}

\begin{Thm}
\label{thm:smalld}
For every $x\in\T$, $M=N$, $\e=\de$, $\mu\in [0,1)$ and $\nu\in [0,1]$, as $\delta \downarrow 0$,
\begin{equation}
\label{Psizerox}
\Psi_{0,x} = c_N \left[ 
\begin{pmatrix}
 0 \\ 0 \\ 0 \\ \alpha(x) 
\end{pmatrix} + \de 
\begin{pmatrix} 
0 \\  \beta(x)\\ \beta(x) \\ 2 c_N \alpha(x)\gamma(0) - 2 \gamma(x) 
\end{pmatrix} \right] + O(\de^2),
\end{equation}
where
\begin{equation}
c_N = \frac{1}{N+\alpha(0)}.
\end{equation}
In particular,
\begin{equation}
\label{prefactor}
\frac{1-\Psi^{(4)}_{0,0}}{N} 
= c_N\left[1+ \delta\,\frac{2\alpha(0)\gamma(0)}{N} + O(\delta^2)\right].
\end{equation}
\end{Thm}

\begin{proof}
Taylor expand in $\de$ the expressions in \eqref{fourierPsialt} and \eqref{Psifouriden}, and 
take the inverse Fourier transform.
\end{proof}

\begin{Remark}
{\rm The first vector in the right-hand side \eqref{Psizerox} must have all entries in $[0,1]$ 
because $\Psi_{0,x}$ is a probability. However, the second vector may in principle take 
values in $\R$ because the seed-bank has a tendency to increase \emph{and} decrease 
$\Psi_{0,x}$. Indeed, while in the seed-bank, individuals are dormant and continue to 
mutate, yet at the same time are more easily traced by other individuals. Thus, there
are two competing effects, exemplified by the different signs in the fourth entry of the 
second vector.}
\end{Remark}

We next compute $\alpha(x)$, $\beta(x)$, $\gamma(x)$ for the uniform nearest-neighbour
model in \eqref{qkernel}. For $x \in \T$ and $l\in\N_0$, let $q_l(x)$ be the probability 
that simple random walk starting from the origin is at site $x$ at time $l$. The Green 
function of simple random walk at site $x$ is
\begin{equation}
G_x(z) = \sum_{l\in\N_0} q_l(x)z^l, \qquad |z| \leq 1.
\end{equation}

\begin{Thm}
\label{thm:smalldalt}
For $q$ as in \eqref{qkernel},
\begin{equation}
\label{abccomp}
\begin{aligned}
\alpha(x) &= \frac{1}{2(1-b)}\,G_x\left(\frac{a}{1-b}\right) 
+ \frac{1}{2(1+b)}\,G_x\left(-\frac{a}{1+b}\right) -\delta_{0,x},\\[0.2cm]
\beta(x) &= \frac{1-\mu}{2\mu}\,\frac{1}{1-b}\,G_x\left(\frac{a}{1-b}\right)
- \frac{1-\mu}{2(2-\mu)}\,\frac{1}{1+b}\,G_x\left(-\frac{a}{1+b}\right)\\
&\qquad -\frac{1}{1-(1-\mu)^2}\,\frac{1}{1-(1-\mu)b}\,
G_x\left(\frac{(1-\mu)a}{1-(1-\mu)b}\right) + \delta_{0,x},\\
\gamma(x) &= \frac{b}{4(1-b)^2}\,G_x\left(\frac{a}{1-b}\right)
+ \frac{a}{4(1-b)^3}\,G'_x\left(\frac{a}{1-b}\right)\\
&\qquad - \frac{b}{4(1+b)^2}\,G_x\left(-\frac{a}{1+b}\right)
- \frac{a}{4(1+b)^3}\,G'_x\left(-\frac{a}{1+b}\right)
\end{aligned}
\end{equation}
with $a=(1-\mu)\nu$, $b=(1-\mu)(1-\nu)$ and $G'_x$ the derivative of $G_x$.
\end{Thm}

\begin{proof}
From \eqref{transkernel} we have $\hat{p}(\theta) = (1-\nu) + \nu \hat{q}(\theta)$. 
Substituting this into the first sum in \eqref{abcFinv}, we get
\begin{equation}
\label{comp1}
\begin{aligned}
\delta_{0,x}+\alpha(x)
&= \sum_{k\in\N_0} m^k \frac{1}{|\hat{\T}|} \sum_{\theta\in\hat{\T}} 
\hat{p}(\theta)^{2k}\,e^{-2\pi i(\theta\cdot x)}\\
&= \sum_{k\in\N_0} m^k \sum_{l=0}^{2k} \binom{2k}{l} (1-\nu)^{2k-l} \nu^l 
\frac{1}{|\hat{\T}|} \sum_{\theta\in\hat{\T}} \hat{q}(\theta)^{l}\,e^{-2\pi i(\theta\cdot x)}\\
&= \sum_{k\in\N_0} m^k \sum_{l=0}^{2k} \binom{2k}{l} (1-\nu)^{2k-l} \nu^l q_l(x)\\
&= \sum_{l\in\N_0} q_l(x)\,a^l
\sum_{k= \lceil l/2 \rceil}^\infty \binom{2k}{l} b^{2k-l}\\
&= \sum_{l\in\N_0} q_l(x)\,a^l 
\sum_{k'=l}^\infty \binom{k'}{l} b^{k'-l} \tfrac12[(+1)^{k'}+(-1)^{k'}]\\
&= \sum_{l\in\N_0} q_l(x)\,a^l\,\tfrac12\left[(+1)^l (1-b)^{-l-1} +(-1)^l (1+b)^{-l-1}\right]\\
&= \frac{1}{2(1-b)}\,G_x\left(\frac{a}{1-b}\right) + \frac{1}{2(1+b)}\,G_x\left(-\frac{a}{1+b}\right), 
\end{aligned}
\end{equation}
which gives the formula for $\alpha(x)$. Next, define 
\begin{equation}
\hat\delta(\theta) = \frac{m\hat{p}(\theta)}{1-m\hat{p}(\theta)}
\end{equation}
and write the second term in \eqref{abcF} as $\hat\beta(\theta)=\hat\alpha(\theta)
\hat\delta(\theta)$. Then the second sum in \eqref{abcFinv} becomes
\begin{equation}
\label{conv}
\beta(x) = \sum_{y \in \T} \alpha(x-y)\delta(y).
\end{equation}
A computation similar to \eqref{comp1} yields
\begin{equation}
\label{comp2}
\begin{aligned}
\delta_{0,x}+\delta(x) 
&= \sum_{k\in\N_0} m^k \frac{1}{|\hat{\T}|} \sum_{\theta\in\hat{\T}} 
\hat{p}(\theta)^k\,e^{-2\pi i(\theta\cdot x)}\\
&= \sum_{k\in\N_0} m^k \sum_{l=0}^k \binom{k}{l} (1-\nu)^{k-l} \nu^l 
\frac{1}{|\hat{\T}|} \sum_{\theta\in\hat{\T}} \hat{q}(\theta)^{l}\,e^{-2\pi i(\theta\cdot x)}\\
&= \sum_{k\in\N_0} m^k \sum_{l=0}^k \binom{k}{l} (1-\nu)^{k-l} \nu^l q_l(x)\\
&= \sum_{l\in\N_0} q_l(x)\,a'^l
\sum_{k= l}^\infty \binom{k}{l} b'^{k-l}\\
&= \sum_{l\in\N_0} q_l(x)\,a'^l\,(1-b')^{-l-1}\\
&= \frac{1}{1-b'}\,G_x\left(\frac{a'}{1-b'}\right)
\end{aligned}
\end{equation}
with $a'=(1-\mu)^2\nu=(1-\mu)a$ and $b'=(1-\mu)^2(1-\nu)=(1-\mu)b$. Moreover, for any 
$|z|,|z'| \leq 1$ with $z \neq z'$ we have the identities
\begin{equation}
\label{gfid}
\begin{aligned}
&\sum_{x\in\T} G_x(z) = \sum_{l\in\N_0} \sum_{x\in\T}z^l q_l(x) 
= \sum_{l\in\N_0} z^l = (1-z)^{-1},\\[0.2cm]
&\sum_{y \in \T} G_{x-y}(z)G_y(z') 
= \sum_{l,l'\in\N_0} z^l z'^{l'} \sum_{y \in \T} q_l(x-y)q_{l'}(y)
= \sum_{l,l'\in\N_0} z^l z'^{l'} q_{l+l'}(x)\\
&= \sum_{k\in\N_0} q_k(x) \sum_{ {l,l'\in\N_0} \atop {l+l'=k} } z^l z'^{l'}
= \sum_{k\in\N_0} q_k(x) z^k\,\frac{1-(\frac{z'}{z})^{k+1}}{1-(\frac{z'}{z})}
= \frac{1}{z-z'} \big[zG_x(z) - z'G_x(z')\big].
\end{aligned}
\end{equation}
Inserting \eqref{comp1} and \eqref{comp2} into \eqref{conv} and using \eqref{gfid}, we 
get the formula for $\beta(x)$ after a short computation. Finally, note from \eqref{abcF} 
that $\hat\gamma(\theta) = m \frac{\partial\hat\alpha(\theta)}{\partial m}$, $\theta\in\hat\T$. 
After substitution into the third sum in \eqref{abcFinv} this gives $\gamma(x) = m 
\frac{\partial\alpha(x)}{\partial m} = \tfrac12(1-\mu)\frac{\partial\alpha(x)}{\partial (1-\mu)}$, 
$x\in\T$. Inserting the formula for $\alpha(x)$, we get the formula for $\gamma(x)$ after 
a short computation.
\end{proof}

By combining Theorems~\ref{thm:smalld} and \ref{thm:smalldalt}, we obtain a formula
for $\Psi_{0,x}$ (up to leading order in $\delta$ for $\delta\downarrow 0$) in terms of the 
Green function of simple random walk. The latter has been studied extensively in the 
literature. We recall the relevant formulas.


\subsection{Green functions}
\label{gf}

The following properties are collected from Montroll~\cite{M64}, \cite{M69}, Montroll and 
Weiss~\cite{MW65}, Spitzer~\cite[Sections I.1, III.15]{S76}, den Hollander and 
Kasteleyn~\cite{dHK82}, Hughes~\cite[Section 3.3, Appendix A.1]{H95}, Abramowitz
and Stegun~\cite[Items 9.6.12, 9.6.13,9.7.2]{AS72}.

\paragraph{$\bullet$ Infinite torus.}
For $L=\infty$, we have $\T=\Z^d$. For $d=1$, the Green function is known in 
closed form: 
\begin{equation}
\label{G1}
G_x(z) = y(z)^{|x|}\,(1-z^2)^{-1/2}, \qquad x \in \Z,\, 0<|z|<1,
\end{equation}
where $y(z) = [1-(1-z^2)^{1/2}]/z$. For $d \geq 2$ no closed form expression is available, 
but there are asymptotic formulas for $z \uparrow 1$. For $d=2$, 
\begin{equation}
\label{G2}
G_x(z) = \frac{1}{\pi} \left[\log\left(\frac{1}{1-z}\right) - C(x) 
+ \bar{C}(x)(1-z)\log\left(\frac{1}{1-z}\right) 
+ O(1-z)\right], \qquad x \in \Z^2,
\end{equation}
where 
\begin{equation}
C(0)=0, \qquad \bar{C}(0)=-\tfrac12,
\end{equation}
and
\begin{equation}
\label{xscal}
\begin{aligned}
&C(x) = \log \|x\|^2 + (2\gamma + \log 8) + O\left(\frac{1}{\|x\|^2}\right),
\qquad \bar{C}(x) = \|x\|^2 + O(1),\\
&\|x\|\to\infty,  \quad \|x\| = \sqrt{x_1^2+x_2^2},
\end{aligned}
\end{equation} 
with $\gamma=0.57721\ldots$ Euler's constant. In particular, $C(k,k) = 4\sum_{l=1}^{|k|} 
\frac{1}{2l-1}$, $k \in \Z$. For $d=3$,
\begin{equation}
\label{G3}
G_x(z) = C(x) - \bar{C}(x) (1-z)^{1/2} + O(1-z), \qquad x \in \Z^3,
\end{equation}
where
\begin{equation}
C(0)= 1.51638\ldots, \qquad \bar{C}(0) = \frac{3\sqrt{3}}{\pi \sqrt{2}},
\qquad \frac{C(x)}{C(0)} = \frac{\bar{C}(x)}{\bar{C}(0)}, \quad x \in \Z^3,
\end{equation}
and
\begin{equation}
\label{xscalalt}
C(x) = \frac{3}{2\pi \|x\|} \left[1+\frac{1}{8\|x\|^2}
\left(-3+5\,\frac{x_1^4+x_2^4+x_3^4}{\|x\|^4}\right)
+ O\left(\frac{1}{\|x\|^4}\right)\right], \qquad \|x\|\to\infty.
\end{equation} 
For $d \geq 4$ the second term in the right-hand side of \eqref{G3} is of higher order. For 
$d=1,2$ the Green functions show diffusive scaling:
\begin{equation}
\label{Gdiff}
\begin{aligned}
&d=1\colon \lim_{z \uparrow 1} \sqrt{2(1-z)}\,G_{y/\sqrt{1-z}}(z) = e^{-\sqrt{2}\,|y|}, 
\quad y \in \R,\\
&d=2\colon \lim_{z \uparrow 1} G_{y/\sqrt{1-z}}(z) = \frac{2}{\pi} K_0(2\|y\|), 
\quad y \in \R^2\setminus \{0\},
\end{aligned}
\end{equation}
where $K_0$ is the modified Bessel function (of the third kind) of order 0. The latter 
satisfies $K_0(u) = [\log(2/u) - \gamma](1+\tfrac14u^2+O(u^4)) + \frac14 u^2 + O(u^4)$,
$u \downarrow 0$, and $K_0(u) = e^{-u} \sqrt{\pi/2u}\,[1-\frac{8}{u}+O(u^{-2})]$, $u\to\infty$. 
The diffusive scaling also holds for the derivative of the Green functions. 

\paragraph{$\bullet$ Finite torus.}
For $L<\infty$, the analogue of \eqref{G1} reads
\begin{equation}
\label{G3fin}
G_x(z) = \frac{y(z)^x+y(z)^{L-x}}{1-y(z)^L}\,(1-z^2)^{-1/2}, 
\qquad x \in \{0,1,\ldots,L-1\}, \, 0<|z|<1,
\end{equation}
while the analogue of \eqref{G2} reads
\begin{equation}
\label{G4}
G_x(z) = L^{-d}(1-z)^{-1} + C_L(x) - \bar{C}_L(x)(1-z) + O((1-z)^2), \qquad x \in \T.
\end{equation} 
The latter expansion also holds for $d=1$. Namely, inserting into \eqref{G3} the expansion 
\begin{equation}
y(z) = 1- [2(1-z)]^{1/2} + \tfrac12 [2(1-z)] - \tfrac38 [2(1-z)]^{3/2} + O((1-z)^2), 
\end{equation}
we find, after a short computation, 
\begin{equation}
\label{xscalalt*}
\begin{aligned}
C_L(x) &= C_L(0) - \frac{x(L-x)}{L},\\ 
\bar{C}_L(x) &= \bar{C}_L(0) - \frac{x(L-x)}{6L}
\left[x(L-x)(L^2-2)-(L^2-5)\right],
\end{aligned}
\end{equation}
with
\begin{equation}
\label{xscalalt**}
C_L(0) = \frac{L^2-1}{6L}, \qquad \bar{C}_L(0) = \frac{(L^2-1)(L^2-19)}{180L}.
\end{equation}
For $d=2$ it is known that $C_L(0) = \frac{2}{\pi} \log L + O(1)$, $\bar{C}_L(0) = c L^2 -\frac{1}{\pi} 
\log L + O(1)$, $L\to\infty$, with $c=0.06187\ldots$, while for $d = 3$ it is known that 
$\lim_{L\to\infty} C_L(0) = C(0)$ and $\bar{C}_L(0)=cL^4[1+o(1)]$ for some $c \in (0,\infty)$.
No formulas are available for $C_L(x)$, $\bar{C}_L(x)$, $x \neq 0$, for $d \geq 2$.  


\subsection{Slower mutation than migration}
\label{wmut}

Return to Theorem~\ref{thm:smalldalt}. Let
\begin{equation}
\rho=\mu/\nu.
\end{equation} 
In terms of this ratio, we have $a = \nu(1-\nu \rho)$ and $b = (1-\nu)(1-\nu \rho)$. We analyse 
what happens in the limit as $\rho \downarrow 0$. To do so, we abbreviate 
\begin{equation}
\label{uvw}
u = \frac{a}{1-b}, \qquad v = \frac{(1-\mu)a}{1-(1-\mu)b}.
\end{equation}
Substitution of \eqref{uvw} into \eqref{abccomp} gives, for $\rho \downarrow 0$ uniformly 
in $\nu$ and $x$,
\begin{equation}
\label{scal}
\begin{aligned}
\alpha(x) &= \frac{1}{2\nu} \left[\frac{1}{1-\nu\rho}\,uG_x(u)\right] + O(1),\\
\beta(x) &= \frac{1}{2\nu^2}\,\frac{1}{\rho} \left[uG_x(u) 
- \frac{1}{(1-\nu\rho)^2 (1-\tfrac12\nu\rho)}\,vG_x(v)\right] + O(1),\\
\gamma(x) &= \frac{1}{4\nu^2}\,\left[\frac{1-\nu}{1+(1-\nu)\rho}\,uG_x(u) 
+ \frac{1}{[1+(1-\nu)\rho]^2}\,uG_x'(u)\right]+ O(1).
\end{aligned}
\end{equation}
Here, the error terms $O(1)$ are valid as long as $\nu$ remains bounded away from 1 (to 
make sure that the singularity of $G_x(z)$ at $z=-1$ does not contribute via the term with
$z=-a/(1+b)$). For $\nu \downarrow 0$ these error terms can be refined to, respectively,
\begin{equation}
-\tfrac34\delta_{0,x}+O(\nu), \quad -\tfrac{1}{16}\delta_{0,x}+O(\nu), 
\quad \tfrac78\delta_{0,x}+O(\nu).
\end{equation} 
To investigate the leading order terms in \eqref{scal}, we expand
\begin{equation}
\label{uvwscal}
\begin{aligned}
u &= 1-\rho + (1-\nu)\rho^2 - (1-\nu)^2\rho^3 + O(\rho^4),\\ 
v &= 1 -2\rho + (4-3\nu)\rho^2 - 4(1-\nu)(2-\nu)\rho^3 + O(\rho^4).
\end{aligned}
\end{equation}
In what follows we derive \emph{expansions in $\rho$ for fixed $\nu$} and keep track of how 
the coefficients in these expansions depend on $\nu$. The various expansions given below 
are \emph{pushed to the lowest order in $\rho$ for which the $x$-dependence becomes visible} 
via the functions $C(x),\bar{C}(x)$ and $C_L(x),\bar{C}_L(x)$ in the Green function formulas in 
Section~\ref{gf}.  The result is a list of formulas for $\alpha(x),\beta(x),\gamma(x)$ that are
technical, but we will see in Section~\ref{reg} that the latter lead to \emph{simple explicit 
scaling expressions} in various interesting limiting regimes.

\paragraph{$\bullet$ Infinite torus.}
For $d=1$, \eqref{G1} and \eqref{scal}--\eqref{uvwscal} give, for $\rho \downarrow 0$ uniformly 
in $\nu$ and $x$,
\begin{equation}
\label{scalabcinf1}
\begin{aligned}
\alpha(x) &= \frac{1}{2\nu}\,\frac{1+O(\rho)}{\sqrt{2\rho}}\,e^{(-\sqrt{2\rho} + O(\rho))\,|x|} + O(1),\\
\beta(x) &= \frac{1}{2\nu^2}\,\frac{1}{\rho} 
\left[\frac{1+O(\rho)}{\sqrt{2\rho}}\,e^{(-\sqrt{2\rho} + O(\rho))\,|x|}
-\frac{1+O(\rho)}{\sqrt{4\rho}}\,e^{(-\sqrt{4\rho} + O(\rho))\,|x|}\right] + O(1),\\
\gamma(x) &= \frac{1}{4\nu^2}\left(\frac{1}{2\rho} +\frac{|x|}{\sqrt{2\rho}}\right)
\frac{1+O(\rho)}{\sqrt{2\rho}}\,e^{(-\sqrt{2\rho} + O(\rho))\,|x|} +O(1), 
\end{aligned}
\end{equation}
where in the last line we use that $G'_x(z) = [\frac{z}{1-z}+ |x| \frac{y'(z)}{y(z)}]G_x(z)$ 
with $\frac{y'(z)}{y(z)} = \frac{1}{z} G_0(z)$. The scaling in \eqref{scalabcinf1} gives us 
control over the dependence on $x$, uniformly in $|x|=O(1/\sqrt{\rho})$. Similarly, for 
$d=2$, \eqref{G2} and \eqref{scal}--\eqref{uvwscal} give, for $\rho \downarrow 0$ 
uniformly in $\nu$ and $x$,
\begin{equation}
\label{scalabcinf2}
\begin{aligned}
\alpha(x) &= \frac{1}{2\pi\nu}\,\bigg[\log\left(\frac{1}{\rho}\right)-C(x) 
+\big[-(1-\nu)+\bar{C}(x)\big]\rho\log\left(\frac{1}{\rho}\right) + O(\rho)\bigg] + O(1),\\
\beta(x) &= \frac{1}{2\pi\nu^2}\,\bigg[\frac{\log 2}{\rho}
-\big[(1-\tfrac52\nu)-\bar{C}(x)\big]\log\left(\frac{1}{\rho}\right) - (1-\tfrac52\nu) C(x)\\
&\qquad\qquad\qquad 
-[(2-\tfrac52\nu)+2 \bar{C}(x)]\log2  + O\left(\rho\log\left(\frac{1}{\rho}\right)\right)\bigg] + O(1),\\
\gamma(x) &= \frac{1}{4\pi\nu^2}\,\bigg[\frac{1}{\rho}
+\big[(1-\nu)-\bar{C}(x)\big] \log\left(\frac{1}{\rho}\right) 
- \big[(3-2\nu)+(1-\nu)C(x)-\bar{C}(x)\big] \\
&\qquad\qquad\qquad 
+O\left(\rho\log\left(\frac{1}{\rho}\right)\right)\bigg] + O(1). 
\end{aligned}
\end{equation}
The scaling in \eqref{scalabcinf2} gives us control over the dependence on $x$, 
uniformly in $\|x\|=o(1/\sqrt{\rho})$ (recall \eqref{xscal}). In order to make sure that the error 
terms $O(1)$ are negligible compared to the terms containing $C(x)$ and $\bar{C}(x)$, 
we need to let $\nu \downarrow 0$ afterwards. For $d = 3$, \eqref{G3} and 
\eqref{scal}--\eqref{uvwscal} give, for $\rho \downarrow 0$ uniformly in $\nu$ and $x$,
\begin{equation}
\label{scalabcinf3}
\begin{aligned}
\alpha(x) &= \frac{1}{2\nu}\,\Big[C(x)-\bar{C}(x)\sqrt{\rho} + O(\rho)\Big] + O(1),\\
\beta(x) &= \frac{1}{2\nu^2}\,\left[\frac{(\sqrt{2}-1)\bar{C}(x)}{\sqrt{\rho}} 
+ (1-\tfrac52\nu)C(x) + O(\sqrt{\rho})\right] + O(1),\\
\gamma(x) &= \frac{1}{4\nu^2}\,\left[ \frac{\bar{C}(x)}{2\sqrt{\rho}} + (1-\nu)C(x) 
+ O(\sqrt{\rho})\right] + O(1). 
\end{aligned}
\end{equation}
The scaling in \eqref{scalabcinf3} gives us control over the dependence on $x$ 
uniformly in $\|x\| = O(1)$ (recall \eqref{xscalalt}), provided we let $\nu \downarrow 0$ 
afterwards. 

\paragraph{$\bullet$ Finite torus.}
For $d \geq 1$, \eqref{G3fin} and \eqref{scal} combine to give, for $\rho \downarrow 0$ 
uniformly in $\nu$ and $x$,
\begin{equation}
\label{scalabcfin}
\begin{aligned}
\alpha(x) &= \frac{1}{2\nu} \left[\frac{1}{L^d \rho} 
+ C_L(x) - \frac{1-\nu}{L^d}+ \left\{\frac{\nu(1-\nu)}{L^d}-(1-\nu)C_L(x)-\bar{C}_L(x)\right\}\rho 
+ O(\rho^2)\right] \hspace{-.1cm} + O(1),\\
\beta(x) &= \frac{1}{2\nu^2} \left[\frac{1}{2L^d \rho^2} - \frac{5 \nu}{4 L^d \rho} 
+ \left\{\frac{\nu(20-17\nu)}{8L^d} + (1-\tfrac52\nu) C_L(x) 
+ \bar{C}_L(x)\right\} + O(\rho)\right] \hspace{-.1cm} + O(1),\\
\gamma(x) &= \frac{1}{4\nu^2} \left[\frac{1}{L^d \rho^2} + \frac{2-\nu}{L^d \rho}
+ \left\{\frac{(2\nu-3)(\nu-1)}{L^d}+(1-\nu)C_L(x)+\bar{C}_L(x)\right\} + O(\rho) \right] 
\hspace{-.1cm} + O(1).
\end{aligned}
\end{equation}
Again these formulas give us sharp control over the dependence on $x$, provided we let
$\nu \downarrow 0$ afterwards. 

\medskip
In Sections~\ref{reg}--\ref{regalt} we use the above expansions to compute $\Psi_{0,x}$, 
$x \in \T$, in the limit as $\delta,\rho,\nu\downarrow 0$ (in that order).  Again, we derive
expansions in $\rho$ whose coefficients depend on $\nu$, and compute these coefficients 
up to order $O(\nu^2)$.


\subsection{Regimes: no seed-bank}
\label{reg}

We first consider the case $\delta=0$, i.e., there is no exchange between the active 
population and the seed-bank. Theorem~\ref{thm:smalld} gives
\begin{equation}
\Psi^{(1)}_{0,x} = \Psi^{(2)}_{0,x} = \Psi^{(3)}_{0,x} = 0, \qquad
\Psi^{(4)}_{0,x} = \frac{\alpha(x)}{N+\alpha(0)}.
\end{equation}
Indeed, if at least one of the two individuals is in the seed-bank, then there cannot be 
coalescence of their lineages, while if both individuals are active, then they remain active 
at all times and behave exactly like in the standard Wright-Fisher model. By inserting
the expansions of $\alpha(x)$ found in \eqref{scalabcinf1}--\eqref{scalabcinf3} and 
\eqref{scalabcfin}, we obtain the following.  

\begin{Thm}
\label{d0exp1}
For $\delta=0$, in the limit as $\rho,\nu \downarrow 0$ (in that order):\\ 
{\rm (1)} For $L=\infty$, 
\begin{eqnarray}
\notag
d=1\colon\hspace{-.5cm}
&&\Psi^{(4)}_{0,x} = \frac{e^{(-\sqrt{2\rho}+O(\sqrt{\rho}))\,|x|} 
- [\frac32\delta_{0,x}\nu+O(\nu^2)] \sqrt{2\rho} + O(\rho)}
{1+[(2N-\frac32)\nu+O(\nu^2)]\sqrt{2\rho}+O(\rho)}\\ \notag
&&\text{uniformly in } |x| = O(1/\sqrt{\rho}),\\ \notag
d=2\colon\hspace{-.5cm}
&&\Psi^{(4)}_{0,x} =\frac{[\log (1/\rho) - [C(x)+\frac32\pi\delta_{0,x}\nu+O(\nu^2)] 
+[-(1-\nu)+\bar{C}(x)]\rho\log(1/\rho)+ O(\rho)}
{\log(1/\rho) + [\pi(2N-\frac32)\nu+O(\nu^2)] + [-(1-\nu)+\bar{C}(0)]\rho\log(1/\rho) + O(\rho)}\\ \notag
&&\text{uniformly in }\|x\| = o(1/\sqrt{\rho}),\\ \notag
d=3\colon\hspace{-.5cm}
&&\Psi^{(4)}_{0,x} = \frac{[C(x)-\frac32\delta_{0,x}\nu+O(\nu^2)] + O(\sqrt{\rho})}
{[C(0)+(2N-\frac32)\nu+O(\nu^2)] + O(\sqrt{\rho})}\\
&&\text{uniformly in } \|x\|=O(1). 
\label{psi4d}
\end{eqnarray}
{\rm (2)} For $L<\infty$,
\begin{equation}
\begin{aligned}
\label{psi4anyd}
d \geq 1\colon\quad
&\Psi^{(4)}_{0,x} = \frac{1+[C_L(x)-\frac32\delta_{0,x}\nu-\frac{1-\nu}{L^d}
+O(\nu^2)]L^d\rho + O(\rho^2)}
{1+[C_L(0)+(2N-\frac32)\nu-\frac{1-\nu}{L^d}+O(\nu^2)]L^d\rho + O(\rho^2)}\\
&\text{uniformly in } x\in\T.
\end{aligned}
\end{equation}
\end{Thm}

The expansions in \eqref{psi4d}--\eqref{psi4anyd} are refinements of what is known from 
the literature (see e.g.\ Durrett~\cite[Chapter 5]{D08}). They simplify in various limiting 
regimes, corresponding to $N,L\to\infty$ in combination with $\rho,\nu\downarrow 0$. 

\begin{Thm}
\label{d0exp2}
For $\delta=0$, in the limit as $\rho,\nu \downarrow 0$ (in that order) and $N,L\to\infty$:\\
{\rm (1)} For $L=\infty$ (recall \eqref{xscal} and \eqref{xscalalt}),
\begin{equation}
\label{scallim1}
\begin{array}{lll}
&d=1,\,N\nu\sqrt{2\rho} \to r \colon 
&\Psi^{(4)}_{0,y/\sqrt{2\rho}} \to \frac{e^{-|y|}}{1+2r}, 
\quad r \in (0,\infty),\,y\in\R,\\[0.2cm]
&d=2,\, \frac{N\nu}{\log(1/\rho)} \to r \colon
&(1-2\chi)\log(1/\rho) \left[\Psi^{(4)}_{0,y/\rho^\chi}-\frac{1-2\chi}{1+2\pi r}\right] 
\to \log(\frac{1}{\|y\|^2}),\\
& &0<\chi<\tfrac12,\, r \in (0,\infty),\,y \in \R^2 \setminus \{0\},\\[0.2cm]
&d=3,\, N\nu \to r \colon 
&\Psi^{(4)}_{0,x} \to \frac{C(x)}{C(0)+2r}, \quad r \in (0,\infty),\, x \in \Z^3,
\end{array}
\end{equation}
and (recall \eqref{Gdiff})
\begin{equation}
\label{scallim2}
d=2,\, \frac{N\nu}{\log(1/\rho)} \to r \colon \quad \log(1/\rho)\,\Psi^{(4)}_{0,y/\sqrt{\rho}}
\to \frac{2 K_0(2\|y\|)}{1+2\pi r}, \quad  r \in (0,\infty),\,y \in \R^2 \setminus \{0\}.
\end{equation}
{\rm (2)} For $L<\infty$ (recall \eqref{xscalalt*}--\eqref{xscalalt**}), 
\begin{equation}
\label{scallim3}
\begin{aligned}
d=1,\,N\nu/L \to r,\, L^2\rho \to s \colon \quad 
&L^{1-\chi} \left[\Psi^{(4)}_{0,yL^\chi} -\frac{1+\frac16 s}{1+(\frac16+2r)s}\right] 
\to \frac{sy}{1+(\frac16+2r)s},\\
&0 < \chi < 1,\,r,s \in (0,\infty),\,y \in [0,\infty). 
\end{aligned}
\end{equation}
\end{Thm}

Theorem~\ref{d0exp2}(1) shows the $x$-dependence of $\Psi^{(4)}_{0,x}$ for $L=\infty$, 
on scale $1/\sqrt{\rho}$ for $d=1$ and on scales $1/\rho^\chi$, $0<\chi \leq \tfrac12$ for 
$d=2$. The dependence is modulated by an $r$-dependent prefactor that is controlled 
by the colony size $N$. The case $r \gg 1$ corresponds to the situation where coalescence 
of the two lineages is slow because colonies are large and migration is fast. The case 
$r \ll 1$ corresponds to the opposite situation where the size of the colonies plays no role. 

Theorem~\ref{d0exp2}(2) does the same for $L<\infty$, on scales $L^\chi$, $0<\chi<1$, for 
$d=1$, modulated by an $(r,s)$-dependent prefactor controlled by the population size $N$ 
and the torus size $L$. The case $s \gg 1$ corresponds to the situation where the boundary 
of the torus is not felt: it takes time $L^2/\nu$ to travel a distance $L$ and time $1/\mu$ for 
a lineage to encounter a mutation. The case $s \ll 1$ corresponds to the situation where the 
system behaves like a homogeneously mixing population, i.e., a single colony with $LN$ 
individuals (``panmictic behaviour'').    

For $L<\infty$, no result is stated for $d \geq 2$ because no formulas are available for 
the coefficients $C_L(x)$, $\bar{C}_L(x)$, $x \neq 0$, of the Green functions.


\subsection{Regimes: slow seed-bank}
\label{regalt}

To monitor the first-order correction in $\delta$, we define
\begin{equation}
\Phi_{0,x} = \lim_{\delta \downarrow 0} \frac{1}{\delta} 
\big[\Psi_{0,x} - \Psi_{0,x}^{\delta=0}\big].
\end{equation}
Theorem~\ref{thm:smalld} gives
\begin{equation}
\Phi^{(1)}_{0,x} = 0, \quad \Phi^{(2)}_{0,x} = \Phi^{(3)}_{0,x} = \frac{\beta(x)}{N+\alpha(0)},
\quad \Phi^{(4)}_{0,x} 
= 2\,\frac{[\alpha(x)\gamma(0)-\alpha(0)\gamma(x)]-N\gamma(x)}{[N+\alpha(0)]^2}.
\end{equation}
By inserting the expansions of $\alpha(x)$ found in \eqref{scalabcinf1}--\eqref{scalabcinf3} 
and \eqref{scalabcfin}, we obtain the following.  

\begin{Thm}
\label{d>0exp1}
In the limit as $\rho,\nu\downarrow 0$ (in that order):\\
{\rm (1)} For $L=\infty$, 
\begin{equation}
\begin{aligned}
&d=1\colon\,\quad \Phi^{(2)}_{0,x} =  \Phi^{(3)}_{0,x}\\ 
&= \frac{1}{\nu\rho}\,\frac{[e^{(-\sqrt{2\rho}+O(\rho))\,|x|}
-\frac{1}{\sqrt{2}}e^{(-\sqrt{4\rho}+O(\rho))\,|x|}]  + O(\rho)}
{1+[(2N-\frac32)\nu+O(\nu^2)]\sqrt{2\rho}+O(\rho)},\\
&d=2\colon\quad \Phi^{(2)}_{0,x} =  \Phi^{(3)}_{0,x}\\ 
&= \frac{1}{\nu\rho}\,\frac{\log 2+[-(1-\frac52\nu)+\bar{C}(x)]\rho\log(1/\rho)+ O(\rho)}
{\log(1/\rho) + [\pi(2N-\frac32)\nu+O(\nu^2)] -[(1-\nu)-\bar{C}(0)]\rho\log(1/\rho) + O(\rho)},\\
&d=3\colon\quad \Phi^{(2)}_{0,x} =  \Phi^{(3)}_{0,x}\\
&= \frac{1}{\nu\sqrt{\rho}}\,\frac{(\sqrt{2}-1)\bar{C}(x) +[ (1-\frac52\nu)C(x) 
-\tfrac18 \nu^2 \de_{0,x} ] \sqrt{\rho} + O(\rho)}
{[C(0)+(2N-\frac32)\nu+O(\nu^2)] -\bar{C}(0)\sqrt{\rho}+O(\rho)},
\end{aligned}
\end{equation}
and
\begin{equation}
\begin{aligned}
&d=1\colon \quad \Phi^{(4)}_{0,x}\\ 
&= \frac{1+O(\rho)}{\nu \sqrt{2\rho}}\,e^{(-\sqrt{2\rho}+O(\rho))\,|x|}
\frac{[|x| +(2N-\tfrac32)\nu+O(\nu^2)] + [(2N-\tfrac32)\nu+O(\nu^2)]|x|\sqrt{2\rho}
+O(\rho)}{[1+[(2N-\tfrac32)\nu+O(\nu^2)] \sqrt{2\rho} + O(\rho)]^2},\\
&d=2\colon \quad \Phi^{(4)}_{0,x}\\ 
&= \frac{1}{\nu} \frac{[-C(x)-\pi(2N-\tfrac32)\nu](1/\rho)-[\bar{C}(0)-\bar{C}(x)] \log^2(1/\rho)
+ O(\log(1/\rho))}{[\log(1/\rho) +[-C(x)+\pi (2N-\tfrac32)\nu+O(\nu^2)]+O(\rho\log(1/\rho))]^2},\\
&d=3\colon \quad \Phi^{(4)}_{0,x}\\ 
&= \frac{1}{2\nu \sqrt{\rho}} \frac{[C(x)\bar{C}(0) -C(0)\bar{C}(x)- (2N -\tfrac32) \nu\bar{C}(x) +O(\nu^2)]
+ O(\sqrt{\rho})}{[[(2N-\tfrac32)\nu+ O(\nu^2)] + O(\sqrt{\rho})]^2}.
\end{aligned}
\end{equation}
{\rm (2)} For $L<\infty$,
\begin{equation}
\begin{aligned}
&d \geq 1\colon \quad \Phi^{(2)}_{0,x} = \Phi^{(3)}_{0,x} = \frac{1}{2\nu\rho} \\
&\times \frac{1-\tfrac52\nu\rho+[ 2C_L(x)+2\bar{C}_L(x)+O(\nu) ]L^d\rho^2+O(\rho^3)}
{1+[C_L(0) -\frac{1-\nu}{L^d} +(2N-\frac32)\nu+O(\nu^2)]L^d\rho 
- [C_L(0)+\bar{C}_L(0) +O(\nu)] L^d \rho^2 + O(\rho^3)}. 
\end{aligned}
\end{equation}
and
\begin{equation}
\begin{aligned}
&d \geq 1\colon\quad \Phi^{(4)}_{0,x} 
= \frac{L^d}{\nu}\,\frac{1}
{[1+[C_L(0)-\frac{1-\nu}{L^d}+(2N-\tfrac32)\nu+O(\nu^2)]L^d\rho + O(\rho^2)]^2} \\
&\times \Big(  [C_L(x)-C_L(0)-(2N-\tfrac32(1-\de_{0,x}))\nu+O(\nu^2)] \\
&+ [-2(\bar{C}_L(x)-\bar{C}_L(0))
+\left(C_(x)-C_L(0)-(2N-\tfrac32(1-\de_{0,x}))\right) \nu +O(\nu^2)] \rho + O(\rho^2)  \Big).
\end{aligned}
\end{equation}
\end{Thm}

Again, various limiting regimes are interesting, in analogy with 
\eqref{scallim1}--\eqref{scallim3}.

\begin{Thm}
\label{d>0exp2}
{\rm (1)}
For $L<\infty$,
\begin{equation}
\begin{array}{lll}
&d=1,\,N\nu\sqrt{2\rho} \to r \colon 
&\nu\rho\,\Phi^{(2)}_{0,y/\sqrt{2\rho}} \to \frac{e^{|y|}-\frac{1}{\sqrt{2}}e^{-\sqrt{2}|y|}}{1+2r},
\quad y \in \R,\\[0.2cm]
&d=2,\, \frac{N\nu}{\log(1/\rho)} \to r \colon
&\nu\rho\,\log(1/\rho)\,\Phi^{(2)}_{0,y/\rho^\chi} \to 
\frac{\log 2}{1+2\pi r}, \quad 0 < \chi < \tfrac12, \quad y \in \R^2,\\[0.2cm]
&d=3,\, N\nu \to r \colon 
&\nu\sqrt{\rho}\,\Phi^{(2)}_{0,x} \to \frac{(\sqrt{2}-1)\bar{C}(x)}{C(0)+2r},
\quad x \in \Z^3,
\end{array}
\end{equation}
and
\begin{equation}
d=2,\, \frac{N\nu}{\log(1/\rho)} \to r \colon \quad
\nu \rho\,\log(1/\rho)\,\Phi^{(2)}_{0,y/\sqrt{\rho}} \to 
\frac{2[K_0(2\|y\|)-K_0(2\sqrt{2}\|y\|)]}{1+2\pi r}, \quad y \in \R^2 \setminus \{0\},
\end{equation}
and
\begin{equation}
\begin{array}{lll}
&d=1,\,N\nu\sqrt{2\rho} \to r \colon 
&2\nu\rho\,\Phi^{(4)}_{0,y/\sqrt{2\rho}} \to e^{-|y|}\,\frac{(1+2r)|y|+2r}{(1+2r)^2},
\quad y \in \R,\\[0.2cm]
&d=2,\, \frac{N\nu}{\log(1/\rho)} \to r \colon
&\nu\rho\,\log(1/\rho)\,\Phi^{(4)}_{0,y/\rho^\chi} \to 
-\frac{2\chi+2\pi r}{(1-2\chi+2\pi r)^2}, \quad 0 < \chi < \tfrac12, \quad y \in \R^2,\\[0.2cm]
&d=3,\, N\nu \to r \colon 
&4 \nu\sqrt{\rho}\,\Phi^{(4)}_{0,x} \to \frac{C(x)\bar{C}(0)-C(0)\bar{C}(x)-2rC(x)}
{(C(0)+2r)^2},
\quad x \in \Z^3,
\end{array}
\end{equation}
and 
\begin{equation}
d=2,\, \frac{N\nu}{\log(1/\rho)} \to r \colon \quad 
\rho\,\log(1/\rho)\,\Phi^{(4)}_{0,y/\sqrt{\rho}} \to - \frac{2\|y\| K_0'(2\|y\|)}{(1+2\pi r)^2},
\quad y \in \R^2 \setminus \{0\}.
\end{equation}
{\rm (2)} 
For $L<\infty$,
\begin{equation}
d=1,\,N\nu/L \to r,\, L^2\rho \to s \colon \quad 2\nu\rho\,\Phi^{(2)}_{0,yL^\chi} 
\to \frac{1}{1+(\frac16+2r)s}, \quad 0 < \chi < 1, \quad y \in [0,\infty).
\end{equation}
and
\begin{equation}
d=1,\,N\nu/L \to r,\, L^2\rho \to s \colon \quad \nu\sqrt{\rho}\,\Phi^{(4)}_{0,yL^\chi} 
\to \frac{r\sqrt{s}}{\frac16+2r}, \quad 0 < \chi < 1, \quad y \in [0,\infty).
\end{equation}
\end{Thm}

\noindent
The same observations as made below Theorem~\ref{d0exp2} apply.


\section{Spatial second moment}
\label{sec:2ndmom}

In Section~\ref{Taylor} we identify the scaling of $\hat{\Psi}(\theta)$ for $\theta\to 0$. Since
$\theta \in \hat{\T} = \{ 0,\frac{1}{L}, \ldots, \frac{L-1}{L}\}^d$, we must require that $L \to \infty$, 
i.e., the torus becomes large. In Section~\ref{2ndFinv} we use this scaling to compute the 
limiting second moment $\zeta=\lim_{L\to\infty} \sum_{x\in\T} |x|^2 \Psi_{0,x}$ in closed form 
for all choices of the parameters, stated in Theorem~\ref{secondmoment}. In Section~\ref{2ndsss} 
we obtain explicit scaling expressions for the special case of the symmetric slow seed-back 
treated in Section~\ref{spec-choice-par}. The value of $\zeta$ is important because it provides 
information about the \emph{average time until the two lineages coalesce}. Namely, we have
\begin{equation}
\label{timecoal}
\lim_{L\to\infty} \mathbb{E}(\tau) = \zeta/2\nu,
\qquad \tau = \text{ time to coalescence}.
\end{equation}
Indeed, on the infinite torus the second moment of simple random walk (with transition kernel
\eqref{qkernel}) after $n\in\N_0$ steps equals $n$. Since $\nu$ is the probability that a lineage 
migrates to a different colony (i.e., makes a move like simple random walk), it follows that the 
second moment of the distance between the two lineages after $n$ steps equals $2\nu n$.  
Picking $n=\tau$ and taking the expectation, we get \eqref{timecoal}.


\subsection{Taylor expansion}
\label{Taylor}

Let
\begin{equation}
\partial \hat{\Psi}(\theta) = \hat{\Psi}(\theta) - \hat{\Psi}(0), \quad
\partial \hat{B}(\theta) = \hat{B}(\theta) - \hat{B}(0), \quad
\partial \hat{\Gamma}(\theta) = \hat{\Gamma}(\theta) - \hat{\Gamma}(0).
\end{equation}
Abbreviate $U= (\mathds{1}-\hat{B}(0))^{-1}$, and let $\Delta(\theta)$ be 
the first term in the Taylor expansion of $\partial\hat{B}(\theta)$:
\begin{equation}
\label{pert}
\partial \hat{B}(\theta) = \Delta(\theta) +o(|\theta|^2).
\end{equation}

\begin{Prop}
\label{proppsipert}
For $\theta\to 0$, 
\begin{align}
\label{Psipert}
\partial \hat{\Psi}(\theta) =  \frac{1-\Psi_{0,0}^{(4)}}{N}\left[ U\, \partial \hat{\Gamma}(\theta)
+ U \, \Delta(\theta) \, U \, \hat{\Gamma}(0)\right] + o(|\theta|^2).
\end{align}
\end{Prop}

\begin{proof}
Using \eqref{pert}, we get
\begin{align}
\big( \mathds{1} -  \hat{B}(\theta)  \big)^{-1} 
&= \sum_{n\in\N_0} \big(\hat{B}(\theta)\big)^n
= \sum_{n\in\N_0} \big( \hat{B}(0) + \Delta(\theta) + o(|\theta|^2) \big)^n \nonumber \\
&= \sum_{n\in\N_0} \big(\hat{B}(0)\big)^n 
+ \sum_{n\in\N_0} \sum_{k=0}^{n-1} \big(\hat{B}(0)\big)^{n-1-k} \Delta(\theta) 
\big(\hat{B}(0)\big)^k + o(|\theta|^2) \nonumber \\
&= \big( \mathds{1} -  \hat{B}(0)  \big)^{-1} 
+ \big( \mathds{1} -  \hat{B}(0)  \big)^{-1} \Delta(\theta) 
\big( \mathds{1} -  \hat{B}(0)  \big)^{-1}+ o(|\theta|^2).
\end{align}
Substitution into \eqref{fourierPsi} yields
\begin{align}
&\partial \hat{\Psi}(\theta) \nonumber \\ 
&= \frac{1-\Psi_{0,0}^{(4)}}{N} \left( 
\big( \mathds{1} -  \hat{B}(0)  \big)^{-1} 
\partial \hat{\Gamma}(\theta) 
+ \big( \mathds{1} -  \hat{B}(0)  \big)^{-1} \Delta(\theta) 
\big( \mathds{1} -  \hat{B}(0)  \big)^{-1} \hat{\Gamma}(\theta) 
\right)+ o(|\theta|^2).
\end{align}
Using that $\hat{\Gamma}(\theta)\to\hat{\Gamma}(0)$ as $\theta\to 0$, we obtain 
\eqref{Psipert}.
\end{proof}

To analyse \eqref{Psipert}, we first compute
\begin{equation}
\label{Umatrix}
U = \big( \mathds{1} -  \hat{B}(0)  \big)^{-1} = \frac{1}{u_0} 
\begin{pmatrix}
u_{1,1} & u_{1,2} & u_{1,3} & u_{1,4} \\
u_{2,1} & u_{2,2} & u_{2,3} & u_{2,4} \\
u_{3,1} & u_{3,2} & u_{3,3} & u_{3,4} \\
u_{4,1} & u_{4,2} & u_{4,3} & u_{4,4}
\end{pmatrix} = \frac{1}{u_0} \begin{pmatrix}
t_1 		& \e s_1 	& \e s_1 	& \e^2 s_2 	\\
\de s_1 	& t_2		& t_3 		& \e s_3	\\
\de s_1	& t_3		& t_2 		& \e s_3	\\
\de^2 s_2	& \de s_3	& \de s_3	& t_4		\\
\end{pmatrix},
\end{equation}
where
\begin{align}
\label{ustdef}
u_0 =&\ (1-m) \left(1-m(1-\de-\e)-2m\de\e\right) (1-m(1-\de-\e)) 
\left(1-m(1-\de-\e)^2 \right), \nonumber \\
s_1 =& \ m (1-m(1-\de-\e)-2m\de\e) (1-\de-m(1-\e)(1-\de-\e)), \nonumber \\ 
s_2 =& \ m (1-m(1-\de-\e)-2m\de\e) (1+m(1-\de-\e)), \nonumber \\
s_3 =& \ m (1-m(1-\de-\e)-2m\de\e) (1-\e-m(1-\e)(1-\de-\e)), \nonumber \\
t_1 =&\ (1-m(1-\de-\e)-2m\de\e)  (1-m(1-\e)^2) (1-m(1-\de-\e)),  \nonumber \\
t_2 =&\  1-m \left[ 3(1-\de-\e)+(\de+\e)^2 -m (1-6\de+4\de^2-\de^3 + 1-6\e+4\e^2-\e^3 \right.\nonumber \\
&\left. +1-\de^3 \e+5\de^2 \e +10\de\e-5\de\e^2+\de\e^3) 
-m^2(1-\de-\e)(\de\e+(1-\de-\e))^2\right], \nonumber \\
t_3 =&\ m^2 \de\e \left[ (1-\e)^2 
+ (1-\de)^2 -2m(1-\de)(1-\e)(1-\de-e) \right], \nonumber \\
t_4 =&\ (1-m(1-\de-\e)-2m\de\e) \left[ (1-m(1-\de)^2) (1-m(1-\de-\e)) -2m\de\e \right].
\end{align}
Note that there are several symmetries in the entries of the matrix in \eqref{Umatrix}, 
which come from the symmetries in the model. Next, for the choice of $p$ in \eqref{transkernel} 
and \eqref{qkernel}, we have 
\begin{align}
\label{pexp}
\hat{p}(\theta) &= \sum_{z\in\T} p(0,z)\, e^{2\pi i (\theta\cdot z)} 
= \sum_{z\in\T} [(1-\nu)\de_{z,0} + \nu q(z)]\, e^{2\pi i (\theta\cdot z)} \nonumber \\
&= (1-\nu) + \nu \hat{q}(\theta) = 1 -\nu (1-\hat{q}(\theta)) = 1 - \nu\pi^2 |\theta|^2 + o(|\theta|^2),
\end{align}
where we use that
\begin{align}
\hat{q}(\theta) = 1 - \pi^2 |\theta|^2 + o(|\theta|^2).
\end{align}
When we substitute \eqref{pexp} into \eqref{Bhatrep}, and use \eqref{Cmatrix} and \eqref{Dhatrep}, 
we get
\begin{equation}
\label{Deltamatrix}
\Delta(\theta) = -\Delta_0 \nu \pi^2 |\theta|^2   + o(|\theta|^2)
\end{equation}
with
\begin{equation}
\label{Delta0matrix}
\Delta_0 = m  (1-\epsilon) 
\begin{pmatrix}
0 & 0                          			& 0                          				& 0 \\
0 & 1-\delta 			  		& 0                          				& \epsilon \\
0 & 0                          			& 1-\delta					& \epsilon\\
0 & \delta 					& \delta     					& 2 (1-\epsilon) 
\end{pmatrix}.
\end{equation}
Finally, recalling \eqref{Gamma} we obtain $\hat{\Gamma}(0)=\Gamma_0$ with
\begin{equation}
\label{Gamma0matrix}
\Gamma_0 = (1-\mu)^2 (1-\e)^2 \begin{pmatrix} 0 \\ 0 \\ 0 \\ 1 \end{pmatrix}
\end{equation}
and
\begin{equation}
\label{gammahatzero}
\partial \hat{\Gamma}(\theta) = -(1-\mu)^2 (1-\e)^2 
\begin{pmatrix} 0 \\ 0 \\ 0 \\ 1-\hat{p}(\theta)^2 \end{pmatrix}
= -2 \, \Gamma_0 \, \nu \pi^2 |\theta|^2 + o(|\theta|^2).
\end{equation}

Combining \eqref{Psipert}, \eqref{Umatrix}, \eqref{Deltamatrix} and \eqref{Psipert}, \eqref{gammahatzero}, 
we arrive at the following proposition identifying the asymptotics of $\partial \hat{\Psi}(\theta)$.

\begin{Prop}
\label{zetapert}
For $\theta\to 0$,
\begin{equation}
\label{Psihatexpbasic}
\partial\hat{\Psi}(\theta) = -  \zeta \pi^2 |\theta|^2 +o(|\theta|^2),
\end{equation}
where $\zeta$ is the $4$-vector
\begin{equation}
\label{zeta}
\zeta = \frac{1-\Psi_{0,0}^{(4)}}{N} \, \nu \, U \, ( 2 \, \mathds{1} + \Delta_0 \, U) \, \Gamma_0
\end{equation}
with $U,\Delta_0$ the $(4 \times 4)$-matrices given by \eqref{Umatrix}--\eqref{ustdef} and 
\eqref{Delta0matrix}, $\Gamma_0$ the $4$-vector given by \eqref{Gamma0matrix}, and 
$\Psi_{0,0}^{(4)}$ the number given by \eqref{Psifouriden}.
\end{Prop}


\subsection{Fourier inversion}
\label{2ndFinv}

We next use Propositions~\ref{proppsipert} and \ref{zetapert} to compute 
$\sum_{x\in\T} |x|^2\Psi_{0,x}$. Recalling the definition of the Fourier transform 
given in \eqref{discFourier}, we have for $\theta\to 0$,
\begin{align}
\hat{\Psi}(\theta) =& \sum_{x\in\T} \Psi_{0,x} 
\:e^{2\pi i (x_1 \theta_1 + x_2 \theta_2)} \nonumber \\
 =& \sum_{x\in\T} \Psi_{0,x} \: \cos\big(2\pi (x_1 \theta_1 + x_2 \theta_2) \big) \nonumber \\
 =& \sum_{x\in\T} \Psi_{0,x} \left( 1-\tfrac12\,4\pi^2 (x_1 \theta_1 + x_2 \theta_2)^2 \right) 
 + o(|\theta|^2) \nonumber \\
=& \hat{\Psi}(0) - 2\pi^2 \sum_{x\in\T} \Psi_{0,x}  
\big(x_1^2 \theta_1^2 + x_2^2 \theta_2^2 + 2 x_1 x_2  \theta_1 \theta_2\big)  + o(|\theta|^2),
\end{align}
where $x=(x_1,x_2)\in\T$ and $\theta=(\theta_1,\theta_2)\in\hat{\T}$. Thus,
\begin{align}
\label{Psihatexp}
\partial\hat{\Psi}(\theta)
=& - 2\pi^2 \sum_{x\in\T} \Psi_{0,x}  \big(x_1^2 \theta_1^2 + x_2^2 \theta_2^2 
+2  x_1 x_2  \theta_1 \theta_2\big)  + o(|\theta|^2).
\end{align}
Choosing $\theta_1 = \theta_2$, and combining \eqref{Psihatexpbasic} and \eqref{Psihatexp}, 
we find
\begin{equation}
\zeta = \sum_{x\in\T}  (x_1^2 + x_2^2+ 2 x_1 x_2  ) \:\Psi_{0,x}.
\end{equation}
We have $\sum_{x\in\T} x_1x_2\Psi_{0,x} = 0$ when $L$ is odd and $\sum_{x\in\T} x_1x_2
\Psi_{0,x} = |\T|\Psi_{0,z_L}$ when $L$ is odd, with $z_L$ one of the corner points of $\T$. 
But $\Psi_{0,z_L} \leq (1-\mu)^{2L}$ and so this term is negligible in the limit as $L\to\infty$. 
We therefore arrive at the following result.

\begin{Thm}
\label{secondmoment}
For all $\de, \e, \mu, N$,
\begin{equation}
\label{zetaiden}
\lim_{L \to \infty} \sum_{x\in\T} |x|^2 \:\Psi_{0,x} = \zeta
\end{equation}
with $\zeta$ given by \eqref{zeta}.
\end{Thm}

Thus, we have found an explicit formula for the second moment in the limit of a large
torus, valid for any choice of the parameters.


\subsection{Symmetric slow seed-bank}
\label{2ndsss}

We close by investigating how $\zeta$ behaves when we return to the special case of the 
symmetric slow seed-bank treated in Section~\ref{spec-choice-par}. In that case we find
(recall \eqref{Umatrix}--\eqref{ustdef})
\begin{equation} 
\label{Uexpansion}
U = c\,\mathds{1} + \de\,m c^2
\begin{pmatrix}
-2 & 1	& 1 & 0 \\
1 & -2	& 0 & 1 \\
1 & 0	& -2 & 1\\
0 & 1 & 1 & -2
\end{pmatrix} + O(\de^2), \quad c= \tfrac{1}{1-m},\,m = (1-\mu)^2,
\end{equation}
and (recall \eqref{Delta0matrix}--\eqref{Gamma0matrix})
\begin{eqnarray}
\Delta_0 &=& m 
\left\{
\begin{pmatrix}
0 & 0 & 0 & 0 \\
0 & 1	 & 0 & 0 \\
0 & 0 & 1 & 0\\
0 & 0	 & 0 &2 
\end{pmatrix}
+ \delta
\begin{pmatrix}
0 & 0 & 0 & 0 \\
0 & -2 & 0 & 1 \\
0 & 0 & -2 & 1\\
0 & 1 & 1 &-4 
\end{pmatrix}
+ O(\delta^2)
\right\},\\
\label{Gammazeropert}
\Gamma_0 &=& m
\left\{ 
\begin{pmatrix} 
0 \\ 0 \\ 0 \\ 1 
\end{pmatrix}
+ \delta  
\begin{pmatrix} 
0 \\ 0 \\ 0 \\ -2 
\end{pmatrix}
+ O(\delta^2)
\right\}.
\end{eqnarray}
After substituting these expressions into \eqref{zeta} and using \eqref{prefactor}, we obtain 
the following.

\begin{Thm}
\label{secondmomentscal}
For $\delta \downarrow 0$,
\begin{align}
\label{zetafirstorderde}
\zeta= c_N\nu\left[1+\delta\,\frac{2\alpha(0)\gamma(0)}{N}+O(\delta^2)\right]\:mc^2
\left\{
\begin{pmatrix} 
0 \\ 0 \\ 0 \\ 2
\end{pmatrix} 
+ \de
\begin{pmatrix} 
0 \\ 3mc \\ 3mc \\ -4 (2c-1)
\end{pmatrix}
+ O\left( \de^2 \right) \right\}.
\end{align}
\end{Thm}

\noindent
After inserting the expansions of $\alpha(0)$ and $\gamma(0)$ found in Section~\ref{wmut},
we obtain expressions for $\zeta$ that are analogous to the expressions for $\Psi_{0,x}$
found in Sections~\ref{reg}--\ref{regalt}. The details are left to the reader.



\end{document}